\documentclass[reqno]{amsart}
\usepackage{bm}
\usepackage{mathrsfs}
\usepackage{amsmath}
\usepackage{amsfonts}
\usepackage{mathtools}
\usepackage{amssymb}
\usepackage{color}
\usepackage{graphicx}
\usepackage{hyperref}
\usepackage{cleveref}
\usepackage{enumitem}
\setlist[enumerate,1]{label=\textup{(\arabic*)}}

\hypersetup{
	colorlinks=true,
	linkcolor=blue,
	filecolor=red,      
	urlcolor=cyan,
	citecolor=green,
}

\DeclareMathOperator{\loc}{loc}
\DeclareMathOperator{\re}{Re}

 \newtheorem{Theorem}{Theorem}[section]
 \newtheorem{Corollary}[Theorem]{Corollary}
 \newtheorem{Lemma}[Theorem]{Lemma}
 \newtheorem{Proposition}[Theorem]{Proposition}

 \newtheorem{Definition}[Theorem]{Definition}

 \newtheorem{Remark}[Theorem]{Remark}

 \numberwithin{equation}{section}

\begin{document}

\title[On the $p$-Bergman kernel with respect to a functional $\xi$]
{On the $p$-Bergman kernel with respect to a functional $\xi$}

\author{Shijie Bao}
\address{Shijie Bao: Institute of Mathematics, Academy of Mathematics and Systems Science, Chinese Academy of Sciences, Beijing 100190, China}
\email{bsjie@amss.ac.cn}

\author{Qi'an Guan}
\address{Qi'an Guan: School of
Mathematical Sciences, Peking University, Beijing 100871, China}
\email{guanqian@math.pku.edu.cn}

\author{Xun Sun}
\address{Xun Sun: Institute of Mathematics, Academy of Mathematics and Systems Science, Chinese Academy of Sciences, Beijing 100190, China}
\email{sunxun@amss.ac.cn}

\thanks{}



\date{}

\dedicatory{}

\commby{}


\begin{abstract}
In the present paper, we generalize the notion of the $p$-Bergman kernel and the $\xi$-Bergman kernel to the $p$-Bergman kernel with respect to a functional $\xi$, and establish some properties of the $p$-Bergman kernel with respect to $\xi$. We also study the relations between the $L^p$ versions of higher order Bergman kernels and $\xi$-Bergman kernels, and as applications we give the reproofs and generalizations of some previous results of B\l ocki and Zwonek about higher order Bergman kernels.
\end{abstract}

\maketitle
\setcounter{tocdepth}{1}
\tableofcontents

\section{Introduction}\label{Introduction}

\subsection{Background}

First we recall the $p$-Bergman kernel and the $\xi$-Bergman kernel, which are on the different directions to generalize the classical Bergman kernel.

\subsubsection{$p$-Bergman kernel}

Let $\Omega\subset\mathbb{C}^n$ be a domain. For any $p\in (0,+\infty)$, the \emph{$p$-Bergman space} is defined to be
\[A^p(\Omega)\coloneqq \left\{ f\in\mathcal{O}(\Omega)\colon \|f\|_p^p=\|f\|^p_{L^p(\Omega)}\coloneqq \int_{\Omega}|f|^p<+\infty\right\},\]
(the integrals are always with respect to the Lebesgue measure in the present paper), and the (on-diagonal) \emph{$p$-Bergman kernel} is defined to be
\[K_{p}(z)\coloneqq \Big(\inf\big\{\|f\|_p^p\colon f\in A^p(\Omega) \ \& \ f(z)=1\big\}\Big)^{-1}, \quad z\in\Omega,\]
which was considered by many authors (e.g. \cite{berndtsson3,dwzz,ns,nzz,sakai,siu,TP18,tsuji}).

Recently, Chen--Zhang established a general $p$-Bergman theory in \cite{CZ12} and obtained many basic properties and results. The \emph{off-diagonal $p$-Bergman kernel} is defined to be
\[K_p(z,w)\coloneqq m_p(z,w)K_p(w)\]
for $p\ge 1$, where $m_p(z,w)$ is defined to be the unique element in
\[\{f\in A^p(\Omega)\colon f(w)=1\}\]
which makes $\int_{\Omega}|f|^p$ minimal. Since then the off-diagonal $p$-Bergman kernel was considered by many authors and many results were obtained (e.g. \cite{chen-xiong1,chen-xiong2,li}).

\subsubsection{$\xi$-Bergman kernel}
We now recall the notion of $\xi$-Bergman kernel. Denote $|\alpha|\coloneqq \sum_{j=1}^n\alpha_j$ for any $\alpha\in \mathbb{N}^n$, where $\alpha_j$ denotes the $j$-th component of $\alpha$. Let $\xi$ be an element in the space
\[\ell_1^{(n)}\coloneqq \left\{\xi=(\xi_{\alpha})_{\alpha\in\mathbb{N}^n}\colon\xi_{\alpha}\in\mathbb{C}, \ \sum_{\alpha\in\mathbb{N}^n}|\xi_{\alpha}|\rho^{|\alpha|}<+\infty \text{ for any } \rho>0\right\}.\]
For any $z_0\in\mathbb{C}^n$, $F(z)\in \mathcal{O}_{z_0}$ and $\xi\in\ell_1^{(n)}$, the value that $\xi$ acts on $F$ is defined as
\[(\xi\cdot F)(z_0)\coloneqq \sum_{\alpha\in\mathbb{N}^n}\xi_{\alpha}\frac{F^{\alpha}(z_0)}{\alpha !}.\]
In fact, $\ell_1^{(n)}$ is the dual space of $\mathcal{O}_{z_0}$ under the analytic Krull topology (see \cite{BG1}).

In \cite{BG1}, Bao--Guan introduced the \emph{$\xi$-Bergman kernel} (also called the \emph{Bergman kernel with respect to $\xi$}), which is defined to be
\[K_{\xi}(z)\coloneqq \sup_{f\in A^{2}(\Omega)}\frac{|(\xi\cdot f)(z)|^2}{\int_{\Omega}|f|^2}.\]
Using the \emph{optimal $L^2$ extension theorem} (cf. \cite{blocki,GZ-L2ext}) and \emph{Guan--Zhou Method}, Bao--Guan generalized Berndtsson's result of the log-plurisubharmonicity property for the fiberwise Bergman kernels (see \cite{berndtsson1}) to the $\xi$-Bergman kernels in \cite{BG1}, and used this property to present a new proof of the sharp eﬀectiveness result of the strong openness property (see also \cite{BG4}). Then Bao--Guan \cite{BG2} obtained the log-plurisubharmonicity property for the fiberwise weighted $\xi$-Bergman kernels, and used this property to deduce a new proof of a sharp eﬀectiveness result of the $L^p$ strong openness property obtained by Guan--Yuan in \cite{GY}.

\subsection{$p$-Bergman theory with respect to $\xi$}

In this article, we are going to develop a general $p$-Bergman theory with respect to $\xi$. We fix a $\xi\in\ell_1^{(n)}$ satisfying that there exists an $\alpha_0\in\mathbb{N}^n$ such that $\xi_{\alpha_0}\neq 0$. For $p>0$, 
\[m_{\xi,p}(z)\coloneqq \inf\big\{\|f\|_p\colon f\in A^p(\Omega) \ \& \ (\xi\cdot f)(z)=1\big\},\]
and for $p\ge 1$, $m_{\xi,p}(\cdot,z)$ is defined to be the unique element in 
\[\big\{f\in A^p(\Omega)\colon(\xi\cdot f)(z)=1\big\}\]
which makes $\int_{\Omega}|f|^p$ minimal (see Proposition \ref{pro1} and Proposition \ref{pro6}).

We define the \emph{(on-diagonal) $p$-Bergman kernel with respect to $\xi$} to be 
\[K_{\xi,p}(z)\coloneqq m_{\xi,p}(z)^{-p}\]
for $p>0$ and the \emph{off-diagonal $p$-Bergman kernel with respect to $\xi$} to be 
\[K_{\xi,p}(\cdot, z)\coloneqq m_{\xi,p}(\cdot,z)m_{\xi,p}(z)^{-p}\]
for $p\ge 1$.

When $\xi=(1,0,\ldots,0,\ldots)$, the $p$-Bergman kernel with respect to $\xi$ and oﬀ-diagonal $p$-Bergman kernel with respect to $\xi$ reduce to the $p$-Bergman kernel and oﬀ-diagonal $p$-Bergman kernel defined in \cite{CZ12}. 

When $p=2$, the $p$-Bergman kernel with respect to $\xi$ reduces to the $\xi$-Bergman kernel defined in \cite{BG1}.

Most of the properties of the $p$-Bergman kernel rely on the basic property that 
for any compact set $K\subset\Omega$ and $p>0$, there exists a constant $C$, such that 
\[\sup_{z\in K} |f(z)|\le C_K \|f\|_p\]
for any $f\in A^p(\Omega)$, which follows directly from the sub-mean value property of $|f|^p$. In \cite{BG1}, Bao--Guan proved that for a fixed $\xi\in\ell_1^{(n)}$, there exists a finite constant $C$ such that 
\begin{align}
\label{a1}
\sup_{z\in K} |(\xi\cdot f)(z)|\le C \|f\|_2
\end{align}
for any holomorphic function $f$ on $\Omega$, and the proof relies heavily on the expansion of $L^2$ integrals. From the sub-mean value property of $|f|^p$ and Cauchy's estimate, we can obtain an estimate of $|f^{(\alpha)}(z_0)|$ in view of $\|f\|_{L^p(B(x_0,r))}$ with suitable coefficients, and generalize \eqref{a1} to all $p\in (0,+\infty)$. 

\begin{Proposition}
\label{pro3}
Let $K$ be a compact subset of $\Omega$. For any $p\in (0,+\infty)$, there exists a constant $C_{K,p}\in (0,+\infty)$, such that
\[\sup_{z\in K} |(\xi\cdot f)(z)|\le C_{K,p} \|f\|_{p}\]
for any holomorphic function $f$ on $\Omega$.
\end{Proposition}

Some basic properties of $p$-Bergman kernel with respect to $\xi$ will be presented, such as $K_{\xi,p}(z)$ is bounded from below by a positive constant and locally bounded from above by a positive constant, and is locally Lipschitz continuous. 

Reproducing formula for the off-diagonal $p$-Bergman kernel established in \cite{CZ12} is one of the most important properties of the off-diagonal $p$-Bergman kernel, and plays a key role in the general theory of $p$-Bergman kernel. The reproducing formula for the off-diagonal $p$-Bergman kernel with respect to $\xi$ also holds by similar methods of calculations of variations. 

\begin{Theorem}
\label{thm1}
For any $z_0\in\Omega$ and $f\in A^p(\Omega)$,
\[(\xi\cdot f)(z_0)=m_{\xi,p}(z_0)^{-p}\int_{\Omega}|m_{\xi,p}(\cdot,z_0)|^{p-2}\overline{m_{\xi,p}(\cdot,z_0)}f.\]
\end{Theorem}

Regularity in a minimizing problem is classic, and regularity properties of the off-diagonal $p$-Bergman kernel and $p$-Bergman kernel were considered and developed in \cite{chen-xiong1,chen-xiong2,CZ12,li}, etc. We can generalize these results to the off-diagonal $p$-Bergman kernel with respect to $\xi$ and $p$-Bergman kernel with respect to $\xi$.

For simplicity, we denote
\[\xi\cdot K_{\xi,p}(z,w)\coloneqq (\xi\cdot K_{\xi,p}(\cdot,w))(z)\]
and
\[\xi\cdot m_{\xi,p}(z,w)\coloneqq (\xi\cdot m_{\xi,p}(\cdot,w))(z)\]
for any $z,w\in\Omega$.

\begin{Theorem}
\label{thm3}
Given any $p>1$ and compact subset $K$ of \ $\Omega$, there exists a constant $C$ such that
\begin{align}
\label{ali1}
|m_{\xi,p}(z,w_1)-m_{\xi,p}(z,w_2)|\le C|w_1-w_2|
\end{align}
and
\begin{align}
\label{ali5}
|\xi\cdot m_{\xi,p}(z,w_1)-\xi\cdot m_{\xi,p}(z,w_2)|\le C|w_1-w_2|
\end{align}
for any $z,w_1,w_2\in K$.
Since $K_{\xi,p}(z)$ is locally Lipschitz continuous (see Proposition \ref{pro4}), the same results hold for $K_{\xi,p}(z,w)$ and $\xi\cdot K_{\xi,p}(z,w)$.
\end{Theorem}

As in the case of the off-diagonal $p$-Bergman kernel, only a slightly weaker result holds for the $1$-Bergman kernel with respect to $\xi$. 

\begin{Theorem}
\label{thm4}
Let $w_1\in\Omega$, and $K\subset\Omega\backslash\{m_{\xi,1}(\cdot,w_1)=0\}$ be a compact set. There exists a constant $C$ such that
\begin{align}
|m_{\xi,1}(z,w_1)-m_{\xi,1}(z,w_2)|\le C|w_1-w_2|\notag
\end{align}
for any $z,w_2\in K$.
\end{Theorem}

For $k\in\mathbb{N}$ and $\alpha\in (0,1]$, denote by $C_{\loc}^{k,\alpha}(\Omega)$ the space of complex-valued functions $u\in C^k(\Omega)$ whose $k$-th order partial derivatives are locally H\"{o}lder
continuous with exponent $\alpha$. Let $(x_1,\ldots,x_{2n})$ be the real coordinate of $\mathbb{R}^{2n}=\mathbb{C}^n$. Then

\begin{Theorem}
\label{thm5}
For any $p> 1$, $K_{\xi,p}\in C_{\loc}^{1,1}(\Omega)$ and
\[\frac{\partial K_{\xi,p}}{\partial x_j}(z)=p\re \frac{\partial (\xi\cdot K_{\xi,p}(\cdot,z))}{\partial x_j}\Bigg|_z\]
for any $z\in \Omega$ and $1\le j\le 2n$.
\end{Theorem}

For $\alpha_1,\alpha_2\in\mathbb{N}^n$, $\alpha_1> \alpha_2$ means that $(\alpha_1)_j\ge (\alpha_2)_j$ for each $j\in \{1,\ldots,n\}$,  \\ and $>$ holds for at least one $j$.

It can be inferred from some basic properties established in this article that $\log K_{\xi,p}(z)$ is a plurisubharmonic function. In further, we can show that $\log K_{\xi,p}(z)$ is strictly plurisubharmonic under the assumption that $p\ge 2$ and there exists an $\alpha_0\in\mathbb{N}^n$ 
such that $\xi_{\alpha_0}\neq 0$ while $\xi_{\alpha}=0$ for all $\alpha>\alpha_0$ with $|\alpha|=|\alpha_0|+1$. 

Since the optimal $L^p$ extension theorem holds for all $p\in (0,2]$ (cf. \cite{GZ-L2ext}), the $p$-Bergman kernel $K_{\xi,p}(z)$ also possesses fiberwise log-plurisubharmonicity when $p\in (0,2]$ by the same method as in \cite{BG1}; see Section \ref{sec-logpsh}.

In \cite{nzz}, Ning--Zhang--Zhou proved a basic difference between the $p$-Bergman kernel and the ordinary Bergman kernel that for any bounded pseudoconvex domain $\Omega$, $K_{p}(z)$ is exhaustive for any $p\in (0,2)$. Similar results hold for the $p$-Bergman kernel with respect to certain $\xi$.

\begin{Theorem}
\label{thm2}
Let $\Omega\subset\mathbb{C}^n$ be a bounded pseudoconvex domain, $0< p< 2$, and $\xi\in \ell_1^{(n)}$ satisfying that $\xi_{\alpha_0}\neq 0$ for an $\alpha_{0}\in \mathbb{N}^{n}$ while $\xi_{\alpha}=0$ for any $\alpha> \alpha_0$. Then 
\[K_{\xi,p}(z)\ge \frac{c}{\delta(z)^p},\]
where $\delta(z)\coloneqq\inf_{w\in\partial\Omega}|z-w|$ and $c$ is a constant independent of $z$.
\end{Theorem}

The same result holds for $p=2$ under the extra assumption that $\Omega$ possesses $C^2$ boundary.

Since $K_{\xi,p}(z)$ is a plurisubharmonic function, a bounded domain is pseudoconvex if and only if $K_{\xi,p}(z)$ is exhaustive for $p\in (0,2)$ and some $\xi$ satisfying the conditions in Theorem \ref{thm2}.

\subsection{Higher order Bergman kernel and $\xi$-Bergman kernel}
We now study the relations between two $L^p$ versions of generalized Bergman kernels, the so-called \emph{higher order Bergman kernels} and the $\xi$-Bergman kernels introduced above. As the applications of the relations and the log-plurisubharmonicity of fiberwise $\xi$-Bergman kernels, we give the reproofs and generalizations of some results of B\l ocki--Zwonek in \cite{BlZw20}.

\subsubsection{Higher order Bergman kernels}
For this type of generalized Bergman kernels, we follow the notations in \cite{BlZw20}. Let
\[H(z)=\sum_{\alpha\in\mathbb{N}^n,\,|\alpha|=k}a_{\alpha}z^{\alpha}\]
be a homogeneous polynomial on $\mathbb{C}^n$ of degree $k$, and define the operator
\[P_H(f)=\sum_{\alpha\in\mathbb{N}^n,\,|\alpha|=k}a_{\alpha}D^{\alpha}f,\]
where $f\in\mathcal{O}(\Omega)$ for some domain $\Omega\subset\mathbb{C}^n$. For any $z\in \Omega$, the higher order Bergman kernel related to $H$ on $\Omega$ is defined as (see \cite{BlZw20}):
\begin{equation*}
    \begin{aligned}
        K^H_{\Omega}(z)\coloneqq \sup\big\{|P_H(f)(z)|^2 \colon f^{(j)}(z)=0, \  j=0,\ldots, k-1,& \\
        f\in A^2(\Omega), \ & \|f\|_{L^2(\Omega)}\le 1\big\}.
    \end{aligned}
\end{equation*}
Here $f^{(j)}(z)$ denotes the $j$-th Fr\'{e}chet derivative of $f$ at $z$. Particularly, $K^1_{\Omega}(z)=K_{\Omega}(z)$ when the polynomial $H=1$.

In the present paper, we also consider the $L^p$ version of the higher order Bergman kernel. Let $p\in (0,+\infty)$, and define the \emph{higher order $p$-Bergman kernel} related to $H$ on $\Omega$ by:
\begin{equation*}
    \begin{aligned}
        K^{H,p}_{\Omega}(z)\coloneqq \sup\big\{|P_H(f)(z)|^p \colon f^{(j)}(z)=0, \  j=0,\ldots, k-1,& \\
        f\in A^p(\Omega), \ & \|f\|_{L^p(\Omega)}\le 1\big\}.
    \end{aligned}
\end{equation*}
We also use the following notation to denote the $p$-Bergman kernels on a domain $\Omega$ with respect to a functional $\xi\in\ell_1^{(n)}$ since we will consider them on different domains:
\begin{equation*}
        K_{\xi,\Omega,p}(z)\coloneqq \sup\big\{|(\xi\cdot f)(z)|^p \colon f\in A^p(\Omega), \ \|f\|_{L^p(\Omega)}\le 1\big\}.
\end{equation*}

\subsubsection{Relations between the generalized Bergman kernels}

For any homogeneous polynomial
\[H(z)=\sum_{\alpha\in\mathbb{N}^n,\,|\alpha|=k}a_{\alpha}z^{\alpha}\]
of degree $k$, we denote $\xi\in \mathbf{S}_H$ if $\xi=(\xi_{\alpha})\in\ell_1^{(n)}$ satisfies:
\begin{equation*}
    \xi_{\alpha}=
\begin{cases}
    a_{\alpha}\cdot \alpha! & \text{if} \ |\alpha|=k, \\
    0 & \text{if} \ |\alpha|>k.\\
\end{cases}
\end{equation*}
It is clear that for any $f\in A^p(\Omega)$ with $f^{(j)}(z)=0$, $j=0,\ldots,k-1$, we have $P_H(f)(z)=(\xi\cdot f)(z)$ for any $\xi\in\mathbf{S}_H$, so
\begin{equation}\label{eq-KHKxi}
    K^{H,p}_{\Omega}(z)\le K_{\xi,\Omega,p}(z), \quad \forall\, \xi\in\mathbf{S}_H.
\end{equation}

Moreover, we have the following result.
\begin{Theorem}\label{thm-inf}
    Let $\Omega$ be a domain in $\mathbb{C}^n$, $z\in \Omega$, and $p\in [1,+\infty)$. Then for any homogeneous polynomial $H$ of degree $k$, we have
    \begin{equation*}
        K^{H,p}_{\Omega}(z)=\inf_{\xi\in\mathbf{S}_H}K_{\xi,\Omega,p}(z)=\min_{\xi\in\mathbf{S}_H}K_{\xi,\Omega,p}(z),
    \end{equation*}
    where the last equality means that the infimum in the middle term can be achieved by some $\xi\in\mathbf{S}_H$.
\end{Theorem}

\subsubsection{Reproofs and generalizations of the results of B\l ocki--Zwonek}
We recall some notations and conventions in \cite{BlZw20}. In the following, we always assume that $\Omega$ is a pseudoconvex domain in $\mathbb{C}^n$ containing the origin $o$. Denote by $G_{\Omega}(\cdot,z)$ the pluricomplex Green function on $\Omega$ with the pole at $z\in \Omega$ and $\Omega_a\coloneqq e^{-a}\{G_{\Omega}(\cdot,o)<a\}$ for each $a\le 0$.

Recall the \emph{Azukawa pseudometric} as follows:
\begin{equation*}
    A_{\Omega}(z;X)\coloneqq \exp \left(\limsup_{\mathbb{C}^*\ni \lambda\to 0}(G_{\Omega}(z+\lambda X,z)-\log|\lambda|)\right),
\end{equation*}
for $z\in \Omega$ and $X\in\mathbb{C}^n$. The \emph{Azukawa indicatrix} at $z$ is
\begin{equation*}
    I_{\Omega}(z)\coloneqq \big\{X\in\mathbb{C}^n \colon A_{\Omega}(z;X)<1\big\}.
\end{equation*}

We will prove the following results.

\begin{Theorem}\label{thm-Kp}
    Let $p\in (0,2]$, and $H$ a homogeneous polynomial on $\mathbb{C}^n$ of degree $k$. For every $\xi\in \mathbf{S}_H$, we have that
    \begin{equation*}
        (-\infty,0]\ni a\rightarrow \log K_{\xi,\{G_{\Omega}(\cdot,o)<a\},p}(o)
    \end{equation*}
    is convex, and
    \[(-\infty,0]\ni a\rightarrow e^{(2n+pk)a}K_{\xi,\{G_{\Omega}(\cdot,o)<a\},p}(o)\]
    is non-decreasing.

    Moreover, for every $\xi\in \mathbf{S}_H$,
\begin{equation*}
    K_{\xi,\Omega,p}(o)\ge\lim_{a\to -\infty}e^{(2n+pk)a}K_{\xi,\{G_{\Omega}(\cdot,o)<a\},p}(o)\ge K_{I_{\Omega}(o)}^{H,p}(o).
\end{equation*}
\end{Theorem}

The following are some corollaries of Theorem \ref{thm-Kp}, where the cases $p=2$ of them were proved by B\l ocki--Zwonek in \cite{BlZw20}.

\begin{Corollary}[see \cite{BlZw20} for $p=2$]\label{cor-BS.nontrivial}
    For $p\in (0,2]$, if the space $A^p(I_{\Omega}(o))$ is non-trivial, then so is $A^p(\Omega)$.
\end{Corollary}

\begin{Corollary}[see \cite{BlZw20} for $p=2$]\label{cor-higher.BK.non.decreasing}
    If $p\in [1,2]$, then the function
    \begin{equation*}
        (-\infty,0]\ni a\rightarrow K_{\Omega_a}^{H,p}(o)
    \end{equation*}
    is non-decreasing.
\end{Corollary}

\begin{Corollary}[see \cite{BlZw20} for $p=2$]\label{cor-BS.infinite.dim}
    Let $p\in [1,2]$. Then $K_{\Omega}^{H,p}(o)\ge K_{I_{\Omega}(o)}^{H,p}(o)$. Moreover, if the space $A^p(I_{\Omega}(o))$ is infinitely dimensional, then so is $A^p(\Omega)$.
\end{Corollary}

\vspace{.1in} {\em Acknowledgements}. The first named author thanks Prof. Zhiwei Wang and Prof. Liyou Zhang for very helpful communications. The second named author was supported by National Key R\&D Program of China 2021YFA1003100, NSFC-11825101, NSFC-11522101 and NSFC-11431013.

\section{Definitions and basic properties}
Let $\Omega\Subset\mathbb{C}^n$ be a bounded domain. The $p$-Bergman space on $\Omega$ is defined to be
\[A^p(\Omega)\coloneqq \left\{ f\in\mathcal{O}(\Omega)\colon \|f\|_p^p=\|f\|^p_{L^p(\Omega)}\coloneqq \int_{\Omega}|f|^p<+\infty\right\}.\]
$A^p(\Omega)$ is a Banach space when $p\ge 1$, and a Hilbert space when $p=2$. For any compact set $S\subset\Omega$ and $p>0$, it can be deduced from the sub-mean value property of $|f|^p$ that 
\begin{align}
\label{eq1}
\sup_S |f|^p\le C_S \|f\|_p^p
\end{align}
for a constant $C_{S}$ and all $f\in A^p(\Omega)$  (see \cite{CZ12}). 

Recall that a Banach space $X$ is said to be \emph{strictly convex} if for any $x_1,x_2\in X$ satisfying that $\|x_1+x_2\|=2\|x_1\|=2\|x_2\|$, $x_1=x_2$ holds.

\begin{Lemma}[{see \cite{CZ12}}]
\label{lemma1}
$A^p(\Omega)$ is strictly convex for $p\ge 1$.
\end{Lemma}

Recall that (cf. \cite{BG1})
\[\ell_1^{(n)}\coloneqq \left\{\xi=(\xi_{\alpha})_{\alpha\in\mathbb{N}^n}\colon\sum_{\alpha\in\mathbb{N}^n}|\xi_{\alpha}|\rho^{|\alpha|}<+\infty \text{ for any} \ \rho>0\right\},\]
where $|\alpha|\coloneqq \sum_{j=1}^n\alpha_j$ for any $\alpha=(\alpha_1,\ldots,\alpha_n)\in\mathbb{N}^n$. Any element $\xi$ in $\ell_1^{(n)}$ is a linear functional on $\mathcal{O}_{z_0}$ for each $z_0\in\mathbb{C}^n$ defined as follows: for any $F(z)=\sum_{\alpha}a_{\alpha}(z-z_0)^{\alpha}\in\mathcal{O}_{z_0}$, 
\[(\xi\cdot F)(z_0)\coloneqq \sum_{\alpha\in\mathbb{N}^n}\xi_{\alpha}\frac{F^{(\alpha)}(z_0)}{\alpha !}=\sum_{\alpha\in\mathbb{N}^n}\xi_{\alpha}a_{\alpha},\]
which is well-defined (cf. \cite{BG1}).

In the following part, we fix a $\xi\in\ell_1^{(n)}$ satisfying $\xi\neq 0$, that is, there exists some $\alpha_0\in\mathbb{N}^n$ such that $\xi_{\alpha_0}\neq 0$.

We firstly recall some lemmas.

\begin{Lemma}[{see \cite[Lemma 2.2]{BG1}}]
\label{lemma2}
For any $f \in\mathcal{O}(\Omega)$, $(\xi\cdot f)\in \mathcal{O}(\Omega)$, that is, $(\xi\cdot f)(z)$ is holomorophic w.r.t. $z\in\Omega$.
\end{Lemma}

\begin{Lemma}[{see \cite[Lemma 2.5]{BG1}}]
\label{lemma3}
Let $\{f_j\}$ be a sequence of holomorphic functions on $\Omega$ which converge uniformly to a holomorphic function $f_0$ on every compact subset of $\Omega$. Then $\{(\xi\cdot f_j)\}$ converge uniformly to $(\xi\cdot f_0)$ on every compact subset of $\Omega$.
\end{Lemma}

Let $z_0\in\Omega$ be a fixed point, and denote
\[m_{\xi,p}(z_0)=m_{\xi,\Omega,p}(z_0)\coloneqq \inf\big\{\|f\|_p\colon f\in A^p(\Omega), \ (\xi\cdot f)(z_0)=1\big\}.\]

\begin{Proposition}
\label{pro1}
There exists at least one element $f_0\in A^{p}(\Omega)$ such that $(\xi\cdot f_0)(z_0)=1$ and $m_{\xi,p}(z_0)=\|f_0\|_p$.
\end{Proposition}

\begin{proof}
Note that $\left(\xi\cdot \frac{(z-z_0)^{\alpha_0}}{\xi_{\alpha_0}}\right)(z_0)=1$, and $\left\|\frac{(z-z_0)^{\alpha_0}}{\xi_{\alpha_0}}\right\|_p<+\infty$ as $\Omega$ is bounded. Thus, $m_{\xi,p}(z_0)<+\infty$.

Let $\{f_j\}$ be a sequence of elements in $A^p(\Omega)$ such that $(\xi\cdot f_j)(z_0)=1$ for every $j$ and $\|f_j\|_p\rightarrow m_{\xi,p}(z_0)$ as $j\rightarrow +\infty$.
Then \eqref{eq1} shows that $\{f_j\}$ forms a normal family on $\Omega$. Therefore, there exists a subsequence $\{f_{j_k}\}$ of $\{f_j\}$ such that $\{f_{j_k}\}$ converges uniformly to a function $f_0\in\mathcal{O}(\Omega)$ on every compact subset of $\Omega$. It follows from Lemma \ref{lemma3} that $(\xi\cdot f_0)(z_0)=1$. In addition, Fatou's Lemma implies
\[\int_{\Omega}|f_0|^p\le \lim_{k\rightarrow +\infty}\int_{\Omega}|f_{j_k}|^p=m_{\xi,p}(z_0)^p.\]
Therefore, $\|f_0\|_p=m_{\xi,p}(z_0)$.
\end{proof}

\begin{Proposition}
\label{pro6}
When $p\ge 1$, there exists a unique $f_0\in A^p(\Omega)$ which satisfies that $(\xi\cdot f_0)(z_0)=1$ and $m_{\xi,p}(z_0)=\|f_0\|_p$.
\end{Proposition}
\begin{proof}
Suppose $f_1, f_2$ are two elements in $A^p(\Omega)$ which satisfy $(\xi\cdot f_j)(z_0)=1$ and $m_{\xi,p}(z_0)=\|f_j\|_p$ for $j=1,2$. On the one hand, we have $\left(\xi\cdot \frac{f_1+f_2}{2}\right)(z_0)=1$, which implies $\left\|\frac{f_1+f_2}{2}\right\|_p\ge m_{\xi,p}(z_0)$ by the definition of $m_{\xi,p}(z_0)$. On the other hand, since $\|f_1\|_p=\|f_2\|_p=m_{\xi,p}(z_0)$, we can see
\[\left\|\frac{f_1+f_2}{2}\right\|^p_p=\int_{\Omega} \left|\frac{f_1+f_2}{2}\right|^p \le \int_{\Omega}\frac{|f_1|^p+|f_2|^p}{2}=m_{\xi,p}(z_0),\]
according to the convexity of the function $x^p$ on $(0,+\infty)$. Therefore, $\left\|\frac{f_1+f_2}{2}\right\|_p=m_{\xi,p}(z_0)$. Lemma \ref{lemma1} shows that $f_1=f_2$.
\end{proof}

For $p\ge 1$, denote by $m_{\xi,p}(\cdot,z_0)$ the unique element in $A^p(\Omega)$ which satisfies the conditions:
\[\big(\xi\cdot m_{\xi,p}(\cdot,z_0)\big)(z_0)=1 \quad \text{and} \quad m_{\xi,p}(z_0)=\|m_{\xi,p}(\cdot,z_0)\|_p.\]

Now we can define the $p$-Bergman kernel on $\Omega$ with respect to some $\xi\in\ell_1^{(n)}$.

\begin{Definition}
Let $\xi\in\ell_1^{(n)}$. For $p>0$, $K_{\xi,\Omega,p}(z)=K_{\xi,p}(z)\coloneqq m_{\xi,p}(z)^{-p}$ is defined to be the $p$-Bergman kernel with respect to $\xi$. 

For $p\ge 1$, $K_{\xi,\Omega,p}(\cdot, z)=K_{\xi,p}(\cdot, z)\coloneqq m_{\xi,p}(\cdot,z)m_{\xi,p}(z)^{-p}$ is defined to be the oﬀ-diagonal $p$-Bergman kernel with respect to $\xi$.
\end{Definition}

It is known that 
\[K_{p}(z)=\sup_{f\in A^p(\Omega)}\frac{|f(z)|^p}{\int_{\Omega}|f|^p},\]
where $K_{p}(z)$ denotes the $p$-Bergman kernel on $\Omega$. Similar result holds for the $p$-Bergman kernel with respect to $\xi$.

\begin{Proposition}
\label{pro2}
For any $p>0$,
\[K_{\xi,p}(z)=\sup_{f\in A^p(\Omega)}\frac{|(\xi\cdot f)(z)|^p}{\int_{\Omega}|f|^p}.\]

\end{Proposition}

\begin{proof}
Let $f_0$ be an element of $A^p(\Omega)$ which satisfies $(\xi\cdot f_0)(z)=1$ and $m_{\xi,p}(z)=\|f_0\|_p$. Since $K_{\xi,p}(z)=m_{\xi,p}(z)^{-p}=\|f_0\|_p^{-p}=1/\int_{\Omega}|f_0|^p$, we have
\[K_{\xi,p}(z)= \frac{|(\xi\cdot f_0)(z)|^p}{\int_{\Omega}|f_0|^p}\le       \sup_{f\in A^p(\Omega)}\frac{|(\xi\cdot f)(z)|^p}{\int_{\Omega}|f|^p}.\]

On the other hand, let $f\in A^p(\Omega)$ be an element such that $(\xi\cdot f)(z)\neq 0$. Then 
\[\left(\xi\cdot\frac{f}{(\xi\cdot f)(z)}\right)(z)=1.\]
Then according to the definition of $m_{\xi,p}(z)$, we have
\[m_{\xi,p}(z)\le \left\| \frac{f}{(\xi\cdot f)(z)} \right\|_p=\frac{1}{|(\xi\cdot f)(z)|}\|f\|_p,\]
which implies that
\[\frac{|(\xi\cdot f)(z)|^p}{\int_{\Omega}|f|^p}=\frac{|(\xi\cdot f)(z)|^p}{\|f\|_p^p}\le m_{\xi,p}(z)^{-p}=K_{\xi,p}(z).\]
Since $f\in A^p(\Omega)$ is arbitrarily chosen with $(\xi\cdot f)(z)\neq 0$, we get
\[\sup_{f\in A^p(\Omega)}\frac{|(\xi\cdot f)(z)|^p}{\int_{\Omega}|f|^p}\le K_{\xi,p}(z).\]
Therefore, Proposition \ref{pro2} is proved.
\end{proof}

Inequality \eqref{eq1} can be generalized to Proposition \ref{pro3}, which plays a very important role in the theory of $p$-Bergman kernel with respect to $\xi$.

\begin{proof}[Proof of Proposition \ref{pro3}]
Let $r\coloneqq \inf_{z\in K}{\text{dist}(z,\partial \Omega)}$, and
\[U\coloneqq \Big\{z\in\Omega\colon \text{dist}(z,K)\le\frac{r}{2}\Big\}.\] 
Then for any $z\in K$, $\overline{B(z,\frac{r}{2})}\subset U$ and for any $z\in U$, $B(z,\frac{r}{2})\subset \Omega$. 

Since $\log|f|$ is a plurisubharmonic function for any holomorphic function $f$,  
\[|f(z)|^p\le\frac{1}{\frac{\pi^n}{n!}\left(\frac{r}{2}\right)^{2n}}\int_{B(z,\frac{r}{2})}|f|^p\le \frac{1}{\frac{\pi^n}{n!}\left(\frac{r}{2}\right)^{2n}}\|f\|_{L^p(\Omega)}^p\]
holds for any $p>0$ and $z\in U$. Therefore,
\begin{align}
\label{b1}
\sup_{U}|f|\le (n!)^\frac{1}{p}\left(\frac{4}{\pi r^2}\right)^{\frac{n}{p}}\|f\|_{L^p(\Omega)}.
\end{align}
For any $z=(z_1,\ldots,z_n)\in K$, Note that
\[\overline{B\left(z_1,\frac{r}{2\sqrt{n}}\right)}\times\cdots\times \overline{B\left(z_n,\frac{r}{2\sqrt{n}}\right)}\subset \overline{B\left(z,\frac{r}{2}\right)}\subset U.\]
It follows from Cauchy's estimate that 
\begin{align}
\label{b2}
\left|\frac{f^{(\alpha)}(z)}{\alpha!}\right|\le \left(\frac{2\sqrt{n}}{r}\right)^{|\alpha|}\sup_{U}|f|
\end{align}
for any $\alpha\in \mathbb{N}^n$ and $z\in K$. Then the inequalities \eqref{b1} and \eqref{b2} imply that
\begin{align}
\label{b3}
\left|\frac{f^{(\alpha)}(z)}{\alpha!}\right|&\le \left(\frac{2\sqrt{n}}{r}\right)^{|\alpha|}(n!)^\frac{1}{p}\left(\frac{4}{\pi r^2}\right)^{\frac{n}{p}}\|f\|_{L^p(\Omega)}
\end{align}
for any $\alpha\in \mathbb{N}^n$ and $z\in K$.

The definition of $\ell_1^{(n)}$ implies that
\[\sum_{\alpha\in\mathbb{N}^n}|\xi_{\alpha}|\left(\frac{2\sqrt{n}}{r}\right)^{|\alpha|}<C_p'\]
for a positive constant $C_p'$. Set $C_{K,p}=C_p'\cdot (n!)^\frac{1}{p}\left(\frac{4}{\pi r^2}\right)^{\frac{n}{p}}$, which is independent of $f$, and we get that
\[|(\xi\cdot f)(z)|\le C_{K,p}\|f\|_{L^p(\Omega)}\]
holds for all $z\in K$. Proposition \ref{pro3} is proved. 
\end{proof}

Next, we generalize some basic properties of the $p$-Bergman kernel to the $p$-Bergman kernel with respect to $\xi$.

\begin{Proposition}
\label{pro10}
$K_{\xi,p}(z)$ is bounded from below by a positive constant and locally bounded from above by a positive constant.
\end{Proposition}

\begin{proof}
Let $R>0$ be the diameter of $\Omega$. Recall that $\xi_{\alpha_0}\neq 0$ for an $\alpha_0\in\mathbb{N}^n$ by assumption. For any $z_0\in\Omega$, we have $(\xi\cdot (z-z_0)^{\alpha_0})(z_0)=\xi_{\alpha_0}$. Then Proposition \ref{pro2} implies that
\[\frac{|\xi_{\alpha_0}|^p}{\int_{B(o,R)}|z^{\alpha_0}|^p}\le \frac{|\xi_{\alpha_0}|^p}{\int_{\Omega}|(z-z_0)^{\alpha_0}|^p}\le K_{\xi,p}(z_0).\]
Note that $\frac{|\xi_{\alpha_0}|^p}{\int_{B(o,R)}|z^{\alpha_0}|^p}$ is a positive constant independent of the choice of $z_0$. Therefore,
\begin{align}
\label{align3}
\frac{|\xi_{\alpha_0}|^p}{\int_{B(o,R)}|z^{\alpha_0}|^p}\le K_{\xi,p}(z)
\end{align}
for all $z\in\Omega$, i.e., $K_{\xi,p}(z)$ is uniformly bounded from below by a positive constant. 

On the other hand, for a compact subset $K$ of $\Omega$, Proposition  \ref{pro2} and Proposition \ref{pro3} imply that
\begin{align}
\label{align4}
K_{\xi,p}(z)\le C
\end{align}
for all $z\in K$, where $C$ is a positive constant depending on $p$ and $\text{dist}(K,\partial \Omega)$. 
\end{proof}

If in further $\xi_{\alpha}=0$ for all $\alpha$ with $|\alpha|>k_0$, then it can be deduced from the proof of Proposition \ref{pro3} that $C_p'= O \left(\frac{1}{r^{\frac{2n}{p}}}\right)$ and $C_p''= O \left(\frac{1}{r^{k_0}}\right)$, which implies
\begin{align}
\label{align5}
K_{\xi,p}(z)\le \frac{C_1}{\delta(z)^{2n+pk_0}}
\end{align}
for all $z\in\Omega$, where $C_1$ is a positive constant independent of $z$ and $\delta(z)=\text{dist}(z,\partial \Omega)$.

\begin{Proposition}
\label{pro4}
$K_{\xi,p}(z)$ is locally Lipschitz continuous for any $p>0$.
\end{Proposition}

\begin{proof}
Let $K$ be a compact subset of $\Omega$. Choose an open set $U$ such that $K\subset U\Subset \Omega$. For any $z_0\in K$, By Proposition \ref{pro1}, there exists some $f_0\in A^p(\Omega)$ such that 
$(\xi\cdot f_{0})(z_0)=\big(K_{\xi,p}(z_0)\big)^{1/p}$ and $\|f_0\|_p=1$.

Proposition \ref{pro3} implies that 
\[\sup_{z\in \overline{U}} |(\xi\cdot f_0)(z)|\le C_p \]
for a constant $C_p>0$ independent of $z_0$ and $f_0$. Since $(\xi\cdot f_0)$ is a holomorphic function by Lemma \ref{lemma2}, we get from Cauchy's estimate that $(\xi\cdot f_0)^{(\alpha)}$ is uniformly bounded on $K$ for a fixed $\alpha\in\mathbb{N}^n$. In particular, there exists a constant $C_K>0$ independent of $z_0$ and $f_0$ such that
\begin{equation*}
    \begin{aligned}
        (K_{\xi,p}(z_0))^{\frac{1}{p}}=|(\xi\cdot f_0)(z_0)|&\le |(\xi\cdot f_0)(z)|+C_K|z-z_0|\\
        &\le (K_{\xi,p}(z))^{\frac{1}{p}}+C_K|z-z_0|
    \end{aligned}
\end{equation*}
for any $z\in K$. As $K_{\xi,p}(z)$ is uniformly bounded on $K$ by   \eqref{align4} and larger than a positive constant by \eqref{align3}, we conclude that $K_{\xi,p}(z)$ is Lipschitz continuous on $K$, which finishes the proof.
\end{proof}

The following results are about $K_{\xi,\Omega,p}$ as $\Omega$ varies.

It can be deduced from Proposition \ref{pro2} that
\begin{align}
\label{align6}
K_{\xi,\Omega,p}(z)&=\sup_{f\in A^p(\Omega)}\frac{|(\xi\cdot f)(z)|^p}{\int_{\Omega}|f|^p}\le  \sup_{f\in A^p(\Omega)}\frac{|(\xi\cdot f)(z)|^p}{\int_{\Omega'}|f|^p} \\                     
&\le \sup_{f\in A^p(\Omega')}\frac{|(\xi\cdot f)(z)|^p}{\int_{\Omega'}|f|^p}=K_{\xi,\Omega',p}(z)\notag
\end{align}
whenever $z\in\Omega'\subset\Omega$. On the other hand, we have
\begin{Proposition}\label{prop-exhaustion}
Let $\Omega_j\subset\Omega\subset\mathbb{C}^n$ be a sequence of bounded domains such that $\Omega_j\subset \Omega_{j+1}$ for $j\ge 1$ and $\bigcup_{j=1}^{+\infty}\Omega_j=\Omega$. Then 
\[\lim_{j\rightarrow +\infty}K_{\xi,\Omega_j,p}(z)=K_{\xi,\Omega,p}(z)\]
locally uniformly on $\Omega$.
\end{Proposition}
\begin{proof}
For any $z_0\in\Omega$, we can choose $j_0$ such that $z_0\in \Omega_{j_0}$ whenever $j\ge j_0$.

By Proposition \ref{pro1}, we can choose an $f_j\in A^{p}(\Omega)$ for each $j\ge j_0$ such that
\[(\xi\cdot f_j)(z_0)=K_{\xi,\Omega_j,p}(z_0)^{1/p} \quad \text{and} \quad \|f_j\|_{L^p(\Omega_j)}=1.\]
It follows from \eqref{align6} that $\lim_{j\rightarrow +\infty}K_{\xi,\Omega_j,p}(z_0)$ exists.

According to Proposition \ref{pro3} and Montel's theorem, we can extract a subsequence $(j_k)$ of $(j)$ such that
$\{f_{j_k}\}$ converges locally uniformly to a holomorophic function $f_0$ on $\Omega$. Lemma \ref{lemma3} implies that 
\[(\xi\cdot f_0)(z_0)=\lim_{k\rightarrow +\infty}(\xi\cdot f_{j_k})(z_0)=\lim_{k\rightarrow +\infty}K_{\xi,\Omega_{j_k},p}(z_0)^{1/p}=\lim_{j\rightarrow +\infty}K_{\xi,\Omega_j,p}(z_0)^{1/p}.\]
On the other hand, Fatou's Lemma induces
\[\|f_0\|^p_{L^p(\Omega)}=\int_{\Omega}|f_0|^p\le \liminf_{k\rightarrow +\infty}\int_{\Omega_{j_k}}|f_{j_k}|^p=1.\]
Therefore, Proposition \ref{pro2} implies
\[K_{\xi,\Omega,p}(z_0)\ge |(\xi\cdot f_0)(z_0)|^p=\lim_{j\rightarrow +\infty}K_{\xi,\Omega_j,p}(z_0).\]
Combining this inequality with \eqref{align6}, we get
\[K_{\xi,\Omega,p}(z_0)=\lim_{j\rightarrow +\infty}K_{\xi,\Omega_j,p}(z_0).\]

Given arbitrary $\epsilon>0$, we choose $j_1$ such that 
\[K_{\xi,\Omega_{j_1},p}(z_0)\le K_{\xi,\Omega,p}(z_0)+\frac{\epsilon}{3}.\]
Choose a neighborhood $U_0$ of $z_0$ such that
\[|K_{\xi,\Omega_{j_1},p}(z)-K_{\xi,\Omega_{j_1},p}(z_0)|<\frac{\epsilon}{3}\]
and
\[|K_{\xi,\Omega,p}(z)-K_{\xi,\Omega,p}(z_0)|<\frac{\epsilon}{3}\]
for any $z\in U_0$. Then
\[K_{\xi,\Omega_{j_1},p}(z)\le K_{\xi,\Omega,p}(z)+\epsilon\]
for any $z\in U_0$.

Since $z_0$ is arbitrarily chosen and $K_{\xi,\Omega_{j},p}(z)$ is decreasing in $j$, we conclude that
\[\lim_{j\rightarrow +\infty}K_{\xi,\Omega_j,p}(z)=K_{\xi,\Omega,p}(z)\]
locally uniformly on $\Omega$.
\end{proof}

Now we consider the situation of product of domains. 

Let $m_1,m_2$ be two positive integers and let $\Omega_1\subset \mathbb{C}^{m_1}$, $\Omega_2\subset \mathbb{C}^{m_2}$ be two domains. For $j=1,2$, let $\xi^{(j)}\in \ell_1^{(m_j)}$. Let $\Omega_0\coloneqq \Omega_1\times\Omega_2\subset \mathbb{C}^{m_1+m_2}$, and $z_0\coloneqq (z_1,z_2)\in \Omega_1\times\Omega_2$. Let $\xi^{(0)}\in \ell_1^{(m_1+m_2)}$ satisfy $\xi^{(0)}_{(\alpha,\beta)}=\xi^{(1)}_{\alpha}\xi^{(2)}_{\beta}$ for all $\alpha\in\mathbb{N}^{m_1}$ and  $\beta\in\mathbb{N}^{m_2}$. We have the following

\begin{Proposition}
\label{pro5}
For any $p>0$,
\begin{align}
\label{align10}
m_{\xi^{(0)},\Omega_0,p}(z_0)=m_{\xi^{(1)},\Omega_1,p}(z_1)\cdot m_{\xi^{(2)},\Omega_2,p}(z_2),
\end{align}
and for any $p\ge 1$,
\begin{align}
\label{align11}
m_{\xi^{(0)},\Omega_0,p}(w_0,z_0)=m_{\xi^{(1)},\Omega_1,p}(w_1,z_1)\cdot m_{\xi^{(2)},\Omega_2,p}(w_2,z_2),
\end{align}
where $w_0=(w_1,w_2)\in \Omega_1\times\Omega_2$.
\end{Proposition}

\begin{proof}
By a change of coordinates, we can assume that $z_0=(o',o'')$, where $o'\in\Omega_1$ and $o''\in\Omega_2$.

Denote the coordinate of $\Omega_0$ by $(z',z'')$, where $z'$ is the coordinate of $\Omega_1$ and $z''$ is the coordinate of $\Omega_2$. Let $f_1\in A^p(\Omega_1)$ and $f_2\in A^p(\Omega_2)$ satisfying that 
\[(\xi^{(1)}\cdot f_1)(o')=1, \quad (\xi^{(2)}\cdot f_2)(o'')=1,\]
and
\[\|f_1\|_{L^p(\Omega_1)}=m_{\xi^{(1)},\Omega_1,p}(o'),\quad \|f_2\|_{L^p(\Omega_2)}=m_{\xi^{(2)},\Omega_2,p}(o'').\]
Suppose that we have the following Taylor's expansions
\[f_1(z')=\sum_{\alpha\in\mathbb{N}^{m_1}}a_{\alpha}(z')^{\alpha}\] 
near $o'$ and 
\[f_2(z'')=\sum_{\beta\in\mathbb{N}^{m_2}}b_{\beta}(z'')^{\beta}\]
near $o''$. Let $f(z',z'')\coloneqq f_1(z')f(z'')$, which is a holomorphic function on $\Omega_0$ such that
\[f(z',z'')=\sum_{\alpha\in\mathbb{N}^{m_1},\,\beta\in\mathbb{N}^{m_2}}a_{\alpha}b_{\beta}(z')^{\alpha}(z'')^{\beta}\]
near $(o',o'')$. Therefore,
\begin{align}
(\xi^{(0)}\cdot f)(o',o'')&=\sum_{\alpha\in\mathbb{N}^{m_1},\beta\in\mathbb{N}^{m_2}}a_{\alpha}b_{\beta}\xi^{(1)}_{\alpha}\xi^{(2)}_{\beta}\notag\\
&=\left(\sum_{\alpha\in\mathbb{N}^{m_1}}a_{\alpha}\xi^{(1)}_{\alpha}\right)\cdot\left(\sum_{\beta\in\mathbb{N}^{m_2}}b_{\beta}\xi^{(2)}_{\beta}\right)=1.\notag
\end{align}
Fubini's theorem shows that
\[\int_{\Omega_0}|f|^p=\int_{\Omega_1}|f_1|^p\int_{\Omega_2}|f_2|^p=m_{\xi^{(1)},\Omega_1,p}(z_1)^pm_{\xi^{(2)},\Omega_2,p}(z_2)^p,\]
yielding that
\[m_{\xi^{(0)},\Omega_0,p}(o',o'')\le m_{\xi^{(1)},\Omega_1,p}(o')\cdot m_{\xi^{(2)},\Omega_2,p}(o'').\]

Next, we prove
\begin{align}
\label{a2}
m_{\xi^{(0)},\Omega_0,p}(o',o'')\ge m_{\xi^{(1)},\Omega_1,p}(o')\cdot m_{\xi^{(2)},\Omega_2,p}(o'').
\end{align}
For any $h\in A^p(\Omega_0)$, set 
\[g(z'')=(\xi^{(1)}\cdot h(\cdot,z''))(o').\] 
By the definition, it can be seen that $g(z'')$ is a holomorphic function on $\Omega_2$. Suppose
\[h(z',z'')=\sum_{\alpha\in\mathbb{N}^{m_1},\beta\in\mathbb{N}^{m_2}}a_{\alpha,\beta}(z')^{\alpha}(z'')^{\beta}\]
near $(o',o'')$. Then
\[g(z'')=\sum_{\beta\in\mathbb{N}^{m_2}}\left(\sum_{\alpha\in\mathbb{N}^{m_1}}a_{\alpha,\beta}\xi^{(1)}_{\alpha}\right)(z'')^{\beta},\]
near $o''$, which implies that
\[(\xi^{(2)}\cdot g)(o'')=\sum_{\alpha\in\mathbb{N}^{m_1},\beta\in\mathbb{N}^{m_2}}a_{\alpha,\beta}\xi^{(1)}_{\alpha}\xi^{(2)}_{\beta}=(\xi^{(0)}\cdot h)(o',o'').\]
According to Proposition \ref{pro2}, we have
\[|(\xi^{(2)}\cdot g)(o'')|^p\le K_{\xi^{(2)},\Omega_2,p}(o'')\int_{\Omega_2}|g|^p,\]
and
\[|g(z'')|^p=|(\xi^{(1)}\cdot h(\cdot,z''))(o')|^p\le K_{\xi^{(1)},\Omega_1,p}(o')\int_{\Omega_1}|h(\cdot,z'')|^p,\]
for any $z''\in\Omega_2$. Hence,
\[|(\xi^{(0)}\cdot h)(o',o'')|^p\le K_{\xi^{(1)},\Omega_1,p}(o')\cdot K_{\xi^{(2)},\Omega_2,p}(o'')\int_{\Omega_0}|h|^p\]
for any $h\in A^p(\Omega_0)$, which implies that 
\[K_{\Omega_0,\xi^{(0)},p}(o',o'')\le K_{\xi^{(1)},\Omega_1,p}(o')\cdot K_{\xi^{(2)},\Omega_2,p}(o''),\]
and then \eqref{a2} follows from $K_{\xi,p}=m_{\xi,p}^{-p}$ for any $p$ and $\xi$.

Let $m(z,z_0)=m_{\xi^{(1)},\Omega_1,p}(w_1,z_1)\cdot m_{\xi^{(2)},\Omega_2,p}(w_2,z_2)$ for $p\ge 1$.
Since 
\[(\xi^{(0)}\cdot m(\cdot,z_0))(z_1,z_2)=1\]
and
\[\|m(\cdot,z_0)\|_{L^p(\Omega_0)}=m_{\xi^{(1)},\Omega_1,p}(z_1)\cdot m_{\xi^{(2)},\Omega_2,p}(z_2),\]
We get that \eqref{align11} holds by Proposition \ref{pro6}.
\end{proof}

When $p=2$, a weighted version of the above proposition was presented in \cite{bgy}.

In the following part of this section, we always assume that $p\ge 1$. 

\begin{Lemma}
\label{lemma5}
For any $z_0\in\Omega$ and $f\in A^p(\Omega)$ with $(\xi\cdot f)(z_0)=0$, we have
\[\int_{\Omega}|m_{\xi,p}(\cdot,z_0)|^{p-2}\overline{m_{\xi,p}(\cdot,z_0)}f=0.\]
\end{Lemma}
\begin{proof}
For $t\in \Delta\coloneqq \{z\in\mathbb{C}\colon|z|<1\}$, set $f_t\coloneqq tf+m_{\xi,p}(\cdot,z_0)$. Note that $f_t$ is a holomorphic function on $\Omega$ satisfying $(\xi\cdot f_t)(z_0)=1$ for any such $t$.

For any $t\in\Delta$, rewrite 
\[|f_t|^p=\left(|m_{\xi,p}(\cdot,z_0)|^2+tf\overline{m_{\xi,p}(\cdot,z_0)}+\overline{tf}m_{\xi,p}(\cdot,z_0)+tf\overline{tf}\right)^\frac{p}{2}.\]
Therefore,
\[\frac{\partial |f_t|^p}{\partial t}=\frac{p}{2}|f_t|^{p-2}(f\overline{m_{\xi,p}(\cdot,z_0)}+f\overline{tf})=\frac{p}{2}|f_t|^{p-2}f\overline{f_t}\]
outside a zero measure set $f_t^{-1}(0)\subset\Omega$. We get
\[\left|\frac{\partial |f_t|^p}{\partial t}\right|\le\frac{p}{2}(|f|+|m_{\xi,p}(\cdot,z_0)|)^{p-1}|f|\]
for every $t\in\Delta$. Note that H\"{o}lder inequality implies
\begin{align}
\label{align7}
\int_{\Omega}(|f|+|m_{\xi,p}(\cdot,z_0)|)^{p-1}|f|\le \||f|+|m_{\xi,p}(\cdot,z_0)|\|_p^{p-1}\|f\|_p<+\infty
\end{align}
when $p>1$, and inequality \eqref{align7} obviously holds for $p=1$. Since $\|f_t\|_p$ takes its minimal value at $t=0$ by the definition of $m_{\xi,p}(\cdot,z_0)$, we conclude from \eqref{align7} and the dominated convergence theorem that
\[0=\frac{\partial \|f_t\|_p^p}{\partial t}\bigg|_{t=0}=\int_{\Omega}\left(\frac{\partial |f_t|^p}{\partial t}\right)\bigg|_{t=0}=\frac{p}{2}\int_{\Omega}|m_{\xi,p}(\cdot,z_0)|^{p-2}\overline{m_{\xi,p}(\cdot,z_0)}f.\]
\end{proof}

When $\xi=(1,0,\ldots,0,\ldots)$, the above result was obtained in \cite{CZ12}. Finally, we prove Theorem \ref{thm1}.

\begin{proof}[Proof of Theorem \ref{thm1}]
Recall that $(\xi\cdot (z-z_0)^{\alpha_0})(z_0)=\xi_{\alpha_0}\neq0$. For any $f\in A^p(\Omega)$, 
\[f(z)-\frac{(\xi\cdot f)(z_0)}{\xi_{\alpha_0}}(z-z_0)^{\alpha_0}\in A^p(\Omega),\]
and
\[\left(\xi\cdot\Big(f(z)-\frac{(\xi\cdot f)(z_0)}{\xi_{\alpha_0}}(z-z_0)^{\alpha_0}\Big)\right)(z_0)=0.\]
Lemma \ref{lemma5} implies that
\[\int_{\Omega}|m_{\xi,p}(\cdot,z_0)|^{p-2}\overline{m_{\xi,p}(\cdot,z_0)}\left(f-\frac{(\xi\cdot f)(z_0)}{\xi_{\alpha_0}}(\cdot-z_0)^{\alpha_0}\right)=0,\]
which can be rewritten as
\begin{align}
\label{align8}
\int_{\Omega}|m_{\xi,p}(\cdot,z_0)|^{p-2}\overline{m_{\xi,p}(\cdot,z_0)}f=\frac{(\xi\cdot f)(z_0)}{\xi_{\alpha_0}}\int_{\Omega}|m_{\xi,p}(\cdot,z_0)|^{p-2}\overline{m_{\xi,p}(\cdot,z_0)}(\cdot-z_0)^{\alpha_0}.
\end{align}
Choose $f=m_{\xi,p}(\cdot,z_0)$. Since $(\xi\cdot m_{\xi,p}(\cdot,z_0))(z_0)=1$, we then get
\begin{align}
\label{align9}
m_{\xi,p}(z_0)^p=\frac{1}{\xi_{\alpha_0}}\int_{\Omega}|m_{\xi,p}(\cdot,z_0)|^{p-2}\overline{m_{\xi,p}(\cdot,z_0)}(\cdot-z_0)^{\alpha_0}.
\end{align}
Combining \eqref{align8} and \eqref{align9}, we get Theorem \ref{thm1}.
\end{proof}

\section{Some regularity results}
Regularity results for the $p$-Bergman kernel were established and developed systematically in \cite{CZ12,chen-xiong1,li,chen-xiong2}. In this section, we mainly generalize these results to the $p$-Bergman kernel with respect to $\xi$.
We fix a $\xi\in\ell_1^{(n)}$ satisfying that there exists an $\alpha_0\in\mathbb{N}^n$ such that $\xi_{\alpha_0}\neq 0$, and $\Omega\Subset\mathbb{C}^n$ denotes a bounded domain. 

The following inequalities arising from the nonlinear analysis of the $p$-Laplacian play a key role in the proof of the regularity results for the $p$-Bergman kernel.

\begin{Lemma}(see \ \cite{lind},\ $\S$12; see also \cite{CZ12})
For any $a,b\in\mathbb{C}$,
\begin{align}
\label{al1}
\rm{Re}\mathnormal{{\{(|b|^{p-\rm{2}}\bar{b}-|a|^{p-\rm2}\bar{a})(b-a)\}}\ge \frac{\rm1}{\rm2}(|b|^{p-\rm2}+|a|^{p-\rm2})|b-a|^{\rm2},\quad p\ge\rm2};
\end{align}
\begin{align}
\label{al2}
&\rm{Re}\mathnormal{{\{(|b|^{p-\rm{2}}\bar{b}-|a|^{p-\rm2}\bar{a})(b-a)\}}}\ge\\
&(p-\rm1)\mathnormal{|b-a|}^{\rm2}(\mathnormal{|a|+|b|})^{\mathnormal{p}-\rm2}+(\rm2-\mathnormal{p})|\rm{Im}(\mathnormal{a\bar{b}})|^{\rm2}(\mathnormal{|a|+|b|})^{\mathnormal{p}-4}, \quad 1\le \mathnormal{p}\le 2;\notag
\end{align}
and
\begin{align}
\label{al3}
\mathnormal{|a|^p\le |b|^p+p\rm{Re}}\mathnormal{\{|a|^{p-\rm2}\bar{a}(a-b)\}, \quad p\ge\rm1}. 
\end{align}
\end{Lemma}

Define
\[H_{\xi,p}(z,w)\coloneqq K_{\xi,p}(z)+K_{\xi,p}(w)-\re \{(\xi\cdot K_{\xi,p}(\cdot,w))(z)+(\xi\cdot K_{\xi,p}(\cdot,z))(w)\}.\]

Choosing $a=m_{\xi,p}(\cdot,w)$ and $b=m_{\xi,p}(\cdot,z)$ in \eqref{al1} and \eqref{al2}, integrations over $\Omega$ and Theorem \ref{thm1} yield the following results.
\begin{Lemma}
\begin{align}
\label{al4}
\int_{\Omega}\frac{|\mathrm{Im}\,(\overline{m_{\xi,1}(\cdot,z)}m_{\xi,1}(\cdot,w))|^2}{(|m_{\xi,1}(\cdot,z)|+|m_{\xi,1}(\cdot,w)|)^3}&\le \frac{1}{K_{\xi,1}(z)K_{\xi,1}(w)}H_{\xi,1}(z,w);\\
\label{al5}
\int_{\Omega} (|m_{\xi,p}(\cdot,z)|+|m_{\xi,p}(\cdot,w)|)^{p-2}&|m_{\xi,p}(\cdot,z)-m_{\xi,p}(\cdot,w)|^2\\
&\le \frac{1}{(p-1)K_{\xi,p}(z)K_{\xi,p}(w)}H_{\xi,p}(z,w), \ \ 1<p\le 2;\notag\\
\label{al6}
\int_{\Omega} (|m_{\xi,p}(\cdot,z)|^{p-2}+|m_{\xi,p}(\cdot,w)|^{p-2})&|m_{\xi,p}(\cdot,z)-m_{\xi,p}(\cdot,w)|^2\\
&\le \frac{2}{K_{\xi,p}(z)K_{\xi,p}(w)}H_{\xi,p}(z,w), \ \ p> 2.     \notag  
\end{align}
\end{Lemma}

Recall that we denote
\[\xi\cdot K_{\xi,p}(z,w)\coloneqq (\xi\cdot K_{\xi,p}(\cdot,w))(z)\]
and
\[\xi\cdot m_{\xi,p}(z,w)\coloneqq (\xi\cdot m_{\xi,p}(\cdot,w))(z)\]
for any $z,w\in\Omega$.
We will need the following estimates for $H_{\xi,p}(z,w)$.
\begin{Lemma}
\label{lemma7}
For any compact set $K\Subset U\subset\Omega$ and $\gamma>0$, where $U$ is an open set, there exists a constant $C_{\gamma}$ such that
\begin{align}
\label{al7}
H_{\xi,p}(z,w)\le C_{\gamma}|z-w|\cdot\|\xi\cdot m_{\xi,p}(\cdot,z)-\xi\cdot m_{\xi,p}(\cdot,w)\|_{L^{\gamma}(U)}
\end{align}
for any $z,w\in K$. There also exists a constant $C'$ such that
\begin{align}
\label{al8}
H_{\xi,p}(z,w)\le C'|z-w|\cdot\sup_{U} |\xi\cdot m_{\xi,p}(\cdot,z)-\xi\cdot m_{\xi,p}(\cdot,w)|
\end{align}
for any $z,w\in K$.
\end{Lemma}
\begin{proof}
For any $z\neq w\in K$,
\begin{align}
&\left|\frac{H_{\xi,p}(z,w)}{z-w}\right|\notag\\
\le&\ \left|\frac{K_{\xi,p}(z)(\xi\cdot m_{\xi,p}(w,w)-\xi\cdot m_{\xi,p}(w,z))-K_{\xi,p}(w)(\xi\cdot m_{\xi,p}(z,w)-\xi\cdot m_{\xi,p}(z,z))}{z-w}\right|        \notag\\
\le&\  K_{\xi,p}(z)\frac{|(\xi\cdot m_{\xi,p}(w,w)-\xi\cdot m_{\xi,p}(w,z))-(\xi\cdot m_{\xi,p}(z,w)-\xi\cdot m_{\xi,p}(z,z))|}{|z-w|}\notag\\
&+\frac{ |K_{\xi,p}(z)- K_{\xi,p}(w)|}{|z-w|}|\xi\cdot m_{\xi,p}(z,w)-\xi\cdot m_{\xi,p}(z,z)|.               \notag
\end{align}

For any $\gamma>0$, inequality \eqref{b3} implies that 
\begin{align}
&\frac{|(\xi\cdot m_{\xi,p}(w,w)-\xi\cdot m_{\xi,p}(w,z))-(\xi\cdot m_{\xi,p}(z,w)-\xi\cdot m_{\xi,p}(z,z))|}{|z-w|}\notag\\
\le&\  C_{\gamma}'\|\xi\cdot m_{\xi,p}(\cdot,z)-\xi\cdot m_{\xi,p}(\cdot,w)\|_{L^{\gamma}(U)}\notag
\end{align}
for a constant $C_{\gamma}'$ and the sub mean-value formula shows that
\[|\xi\cdot m_{\xi,p}(z,z)-\xi\cdot m_{\xi,p}(z,w)|\le C_{\gamma}''\|\xi\cdot m_{\xi,p}(\cdot,z)-\xi\cdot m_{\xi,p}(\cdot,w)\|_{L^{\gamma}(U)}\]
for a constant $C_{\gamma}''$.

The inequality \eqref{align4} and Proposition \ref{pro4} imply that $K_{\xi,p}(z)$ and 
\[\frac{ |K_{\xi,p}(z)- K_{\xi,p}(w)|}{|z-w|}\] 
are bounded by a constant $C'$. Then \eqref{al7} can be proved by combining the above results.

Note that 
\[\|\xi\cdot m_{\xi,p}(\cdot,z)-\xi\cdot m_{\xi,p}(\cdot,w)\|_{L^{\gamma}(U)}\le \mathrm{vol}(U)^{\frac{1}{\gamma}}\sup_{U} |\xi\cdot m_{\xi,p}(\cdot,z)-\xi\cdot m_{\xi,p}(\cdot,w)|.\]
We get \eqref{al8}.
\end{proof}

We will need the following result about the convergence of $m_{\xi,p}(\cdot,w)$.
\begin{Lemma}
\label{lemma8}
Let $p\ge 1$. For any sequence of points $w_j\rightarrow w_0\in\Omega$, there exists a subsequence $j_k$ such that
\[m_{\xi,p}(\cdot,w_{j_k})\rightarrow m_{\xi,p}(\cdot,w_0)\]
locally uniformly on $\Omega$.
\end{Lemma}
\begin{proof}
Since 
\begin{align}
\label{al9}
\int_{\Omega}|m_{\xi,p}(\cdot,w_{j})|^p=\frac{1}{K_{\xi,p}(w_{j})}\le \frac{\int_{B(o,R)}|z^{\alpha_0}|^p}{|\xi_{\alpha_0}|^p}
\end{align}
by \eqref{align3}, we know from the sub mean-value formula that $\big(m_{\xi,p}(\cdot,w_{j})\big)_{j\ge 1}$ forms a normal family. There exists a subsequence $j_k$ such that
$m_{\xi,p}(\cdot,w_{j_k})$ converge locally uniformly to a holomorphic function $m_0(\,\cdot\,)$ on $\Omega$. Fatou's lemma shows that
\[\int_{\Omega}|m_0|^p\le \liminf_{k\rightarrow +\infty}\int_{\Omega}|m_{\xi,p}(\cdot,w_{j_k})|^p=\liminf_{k\rightarrow +\infty} m_{\xi,p}(w_{j_k})^p=m_{\xi,p}(w_0)^p.\]

We can assume that $w_{j_k}\in K$ for a compact subset $K$ of $\Omega$. 
Inequalities \eqref{b3} and \eqref{al9} show
\[\frac{|m_{\xi,p}(w_{j_k},w_{j_k})-m_{\xi,p}(w_0,w_{j_k})|}{|w_{j_k}-w_0|}\le C\]
for a constant $C$, which implies that $m_{\xi,p}(w_0,w_{j_k})\rightarrow 1$ and hence $m_0(w_0)=1$.
We deduce from Proposition \ref{pro6} that $m_0(\cdot)=m_{\xi,p}(\cdot,w_0)$, Lemma \ref{lemma8} has been proved.
\end{proof}

Now we prove Theorem \ref{thm3}.
\begin{proof}[Proof of Theorem \ref{thm3}]
Note that \eqref{align3} and \eqref{align4} imply that 
\[\frac{1}{C_1}\le K_{\xi,p}(w)\le C_1\]
for any $w\in K$, where $C_1$ is a constant. 

First, we assume $1<p\le 2$. H\"{o}lder's inequality and \eqref{al5} show that
\begin{align}
&\int_{\Omega}|m_{\xi,p}(\cdot,w_1)-m_{\xi,p}(\cdot,w_2)|^p\notag\\
=&\ \int_{\Omega}|m_{\xi,p}(\cdot,w_1)-m_{\xi,p}(\cdot,w_2)|^p|m(\cdot,w_1)|^{\frac{(p-2)p}{2}}|m(\cdot,w_1)|^{\frac{(2-p)p}{2}}\notag\\
\le&\ \left(\int_{\Omega}|m_{\xi,p}(\cdot,w_1)-m_{\xi,p}(\cdot,w_2)|^2|m(\cdot,w_1)|^{p-2} \right)^{\frac{p}{2}} \left(\int_{\Omega}|m(\cdot,w_1)|^p\right)^{1-\frac{p}{2}}       \notag\\
\le&\  C_2 H_{\xi,p}(w_1,w_2)^{\frac{p}{2}},\notag
\end{align}
where $C_2$ is a constant. Let $U$ be an open set such that $K\subset U\Subset \Omega$. Combining the above results with \eqref{al7}, we conclude that
\begin{equation}\label{eq4}
    \begin{aligned}
\|m_{\xi,p}(\cdot,w_1)&-m_{\xi,p}(\cdot,w_2)\|_{L^p(\Omega)}\le C_2^{\frac{1}{p}} H_{\xi,p}(w_1,w_2)^{\frac{1}{2}}\\
&\le C_2^{\frac{1}{p}}C_{\gamma}^{\frac{1}{2}}|w_1-w_2|^{\frac{1}{2}}\cdot\|\xi\cdot m_{\xi,p}(\cdot,w_1)-\xi\cdot m_{\xi,p}(\cdot,w_2)\|_{L^p(U)}^{\frac{1}{2}}.
\end{aligned}
\end{equation}

It follows from Proposition \ref{pro3} and inequality \eqref{eq4} that
\begin{align}
&\|m_{\xi,p}(\cdot,w_1)-m_{\xi,p}(\cdot,w_2)\|_{L^p(\Omega)}\notag\\
\le& \ C_2^{\frac{1}{p}}C_{\gamma}^{\frac{1}{2}}|w_1-w_2|^{\frac{1}{2}}\|\xi\cdot m_{\xi,p}(\cdot,w_1)-\xi\cdot m_{\xi,p}(\cdot,w_2)\|_{L^p(U)}^{\frac{1}{2}}\notag\\
\le& \ C_3|w_1-w_2|^{\frac{1}{2}}\|m_{\xi,p}(\cdot,w_1)-m_{\xi,p}(\cdot,w_2)\|_{L^p(\Omega)}^{\frac{1}{2}}\notag
\end{align}
for a constant $C_3$, which implies that
\begin{equation}
\label{ali3}
\|m_{\xi,p}(\cdot,w_1)-m_{\xi,p}(\cdot,w_2)\|_{L^p(\Omega)}\le C_3^2|w_1-w_2|.
\end{equation}
Then \eqref{ali1} follows from the sub mean-value formula.

Proposition \ref{pro3} and inequality \eqref{ali3} imply that
\[\sup_{K}|\xi\cdot m_{\xi,p}(\cdot,w_1)-\xi\cdot m_{\xi,p}(\cdot,w_2)|\le C_3'|w_1-w_2|\]
for a constant $C_3'$. Thus, \eqref{ali5} holds.

Now we assume $p>2$. By Lemma \ref{lemma8}, $\big(m_{\xi,p}(\cdot,w)\big)_{w\in K}$ forms a compact family of holomorphic functions in the topology of locally uniform convergence. Demailly--Koll\'{a}r's theorem on the semi-continuity property of the complex singularity exponent (\cite{demailly4}) implies that
\begin{equation}
\label{ali2}
\int_{U_1}|m_{\xi,p}(\cdot,w)|^{-c}\le M
\end{equation}
for all $w\in K$, where $U_1$ is an open subset of $\Omega$ such that $K\subset U_1$, and $c>0$, $M>0$ are constants.

 Denote $\beta=\frac{2c}{p-2+c}<2$. Then H\"{o}lder inequality, inequalities \eqref{ali2} and \eqref{al6} indicate that
\begin{equation*}
    \begin{aligned}
&\int_{U_1}|m_{\xi,p}(\cdot,w_1)-m_{\xi,p}(\cdot,w_2)|^\beta\\
=&\ \int_{U_1}|m_{\xi,p}(\cdot,w_1)-m_{\xi,p}(\cdot,w_2)|^\beta|m_{\xi,p}(\cdot,w_1)|^{\frac{(p-2)\beta}{2}}|m_{\xi,p}(\cdot,w_1)|^{\frac{(2-p)\beta}{2}}\\
\le&\  \left\{\int_{U_1}|m_{\xi,p}(\cdot,w_1)|^{p-2}|m_{\xi,p}(\cdot,w_1)-m_{\xi,p}(\cdot,w_2)|^2\right\}^{\frac{\beta}{2}}\cdot\left\{\int_{U_1}|m_{\xi,p}(\cdot,w_1)|^{-c}\right\}^{1-\frac{\beta}{2}}\\
\le&\  C_4 H_{\xi,p}(w_1,w_2)^{\frac{\beta}{2}}
\end{aligned}
\end{equation*}
for a constant $C_4$. Take an open set $U_2$ such that 
\[K\subset U_2\Subset U_1.\]
Combining the above results with \eqref{al7} and Proposition \ref{pro3}, we conclude that
\begin{equation*}
\begin{aligned}
\|m_{\xi,p}(\cdot,w_1)&-m_{\xi,p}(\cdot,w_2)\|_{L^\beta(U_1)}\le C_4^{\frac{1}{\beta}}H_{\xi,p}(w_1,w_2)^{\frac{1}{2}}\\
&\le C_4^{\frac{1}{\beta}}C_{\beta}^{\frac{1}{2}}|w_1-w_2|^{\frac{1}{2}}\|\xi\cdot m_{\xi,p}(\cdot,w_1)-\xi\cdot m_{\xi,p}(\cdot,w_2)\|_{L^\beta(U_2)}^{\frac{1}{2}}\\
&\le C_5|w_1-w_2|^{\frac{1}{2}}\|m_{\xi,p}(\cdot,w_1)-m_{\xi,p}(\cdot,w_2)\|_{L^\beta(U_1)}^{\frac{1}{2}}
\end{aligned}
\end{equation*}
for a constant $C_5$, which implies that
\begin{align}
\label{ali6}
\|m_{\xi,p}(\cdot,w_1)-m_{\xi,p}(\cdot,w_2)\|_{L^{\beta}(U_1)}\le C_5^2|w_1-w_2|.
\end{align}
Thus, we get \eqref{ali1} by the sub-mean value formula.

Proposition \ref{pro3} and \eqref{ali6} imply that
\[\sup_{K}|\xi\cdot m_{\xi,p}(\cdot,w_1)-\xi\cdot m_{\xi,p}(\cdot,w_2)|\le C_5'|w_1-w_2|\]
for a constant $C_5'$, which gives \eqref{ali5}.
\end{proof}

Next, we prove Theorem \ref{thm4}.

\begin{proof}[Proof of Theorem \ref{thm4}]
Let $U_1$ be an open set such that $K\subset U_1\Subset \Omega\backslash\{m_{\xi,1}(\cdot,w_1)=0\}$. It follows from \eqref{align3} that $m_{\xi,1}(w)^p=\int_{\Omega}|m_{\xi,1}(\cdot,w)|^p$ is uniformly bounded. Note that $m_{\xi,1}(\cdot,w_1)$ is a non-vanishing continuous function on $\overline{U_1}$. Hence, there exists a constant $C_1$ such that
\begin{equation}
\label{al10}
\frac{1}{C_1}\le |m_{\xi,1}(z,w_1)|\le |m_{\xi,1}(z,w_1)|+|m_{\xi,1}(z,w_2)|\le C_1
\end{equation}
for any $z,w_2\in U_1$.

According to \eqref{al4} and \eqref{al10}, we have
\begin{align}
\int_{U_1}\left|\mathrm{Im}\,\left\{\frac{m_{\xi,1}(\cdot,w_2)}{m_{\xi,1}(\cdot,w_1)}\right\}\right|^2&=\int_{U_1}\left|\mathrm{Im}\,\left\{\frac{m_{\xi,1}(\cdot,w_2)\overline{m_{\xi,1}(\cdot,w_1)}}{|m_{\xi,1}(\cdot,w_1)|^2}\right\}\right|^2  \notag\\ 
&\le C_1^5\int_{U_1}\frac{|\mathrm{Im}\,(\overline{m_{\xi,1}(\cdot,w_1)}m_{\xi,1}(\cdot,w_2))|^2}{(|m_{\xi,1}(\cdot,w_1)|+|m_{\xi,1}(\cdot,w_2)|)^3}\\
&\le C_2H_{\xi,1}(w_1,w_2)\notag
\end{align}
for a constant $C_2$. Let $U_2, U_3, U_4$ be domains such that $K\Subset U_4 \Subset U_3\Subset U_2\Subset U_1$.
It follows from Proposition \ref{pro3}, \eqref{al8} and the sub mean-value formula that
\begin{equation*}
\begin{aligned}     
\sup_{U_2}\left|\mathrm{Im}\,\left\{\frac{m_{\xi,1}(\cdot,w_2)}{m_{\xi,1}(\cdot,w_1)}\right\}\right|&\le C_3\left\{\int_{U_1}\left|\mathrm{Im}\,\left\{\frac{m_{\xi,1}(\cdot,w_2)}{m_{\xi,1}(\cdot,w_1)}\right\}\right|^2\right\}^{\frac{1}{2}}\\
&\le C_3C_2^{\frac{1}{2}}H_{\xi,1}(w_1,w_2)^{\frac{1}{2}}\notag\\
&\le C_3C_2^{\frac{1}{2}}C'^{\frac{1}{2}}|z-w|^{\frac{1}{2}}\cdot\sup_{U_4} |\xi\cdot m_{\xi,p}(\cdot,w_1)-\xi\cdot m_{\xi,p}(\cdot,w_2)|^{\frac{1}{2}}\\
&\le C_4|z-w|^{\frac{1}{2}}\cdot\sup_{U_3} |m_{\xi,p}(\cdot,w_1)-m_{\xi,p}(\cdot,w_2)|^{\frac{1}{2}}
\end{aligned}
\end{equation*}
for a constant $C_4$.

Recall that the Borel--Carath\'{e}odory inequality (see e.g., \cite{CZ12}) shows
\[\sup_{\Delta_r}|f|\le \frac{2r}{R-r}\sup_{\Delta_R}\mathrm{Im}\,f+\frac{R+r}{R-r}|f(0)|\]
for any holomorophic function $f$ on $\Delta_R\coloneqq \{z\in\mathbb{C}\colon|z|<R\}$ and $0<r<R$. Let $r=\frac{1}{2}R$, and it follows that
\begin{equation}\label{ali4}
    \sup_{\Delta_{\frac{R}{2}}}|f|\le 2\sup_{\Delta_R}\mathrm{Im}\,f+3|f(0)|.
\end{equation}

Let $h\coloneqq \frac{m_{\xi,1}(\cdot,w_2)}{m_{\xi,1}(\cdot,w_1)}-1$. Since $m_{\xi,1}(w)=\int_{\Omega}|m_{\xi,1}(\cdot,w)|$ is locally Lipschitz continuous by Proposition \ref{pro4},
it follows from the sub mean-value formula that 
\[|h(w_1)|\le C_5|w_1-w_2|\]
for a constant $C_5$. Since $\{B_{\frac{1}{2}\mathrm{dist}(w,\partial U_2 )}(w):w\in \overline{U_3}\}$ form an open covering of $U_3$, a finite chain of balls lying in $B_{\frac{1}{2}\mathrm{dist}(w,\partial U_2 )}(w)$ for some $w\in \overline{U_3}$ connect $w_1$ to any point in $U_3$.
Applying (\ref{ali4}) successively on each ball, we conclude that
\[\sup_{U_3} |h|\le C_6|w_1-w_2|+C_7|w_1-w_2|^{\frac{1}{2}}\cdot\sup_{U_3} |m_{\xi,p}(\cdot,w_1)-m_{\xi,p}(\cdot,w_2)|^{\frac{1}{2}}\]
for two constants $C_6$ and $C_7$,
which implies that
\begin{align}
&\sup_{U_3} |m_{\xi,p}(\cdot,w_1)-m_{\xi,p}(\cdot,w_2)|\notag\\
&\le C_1C_6|w_1-w_2|+C_1C_7|w_1-w_2|^{\frac{1}{2}}\cdot\sup_{U_3} |m_{\xi,p}(\cdot,w_1)-m_{\xi,p}(\cdot,w_2)|^{\frac{1}{2}}\notag
\end{align}
by \eqref{al10}. As a result,
\begin{align}
\label{ali7}
\sup_{U_3} |m_{\xi,p}(\cdot,w_1)-m_{\xi,p}(\cdot,w_2)|\le C_8|w_1-w_2|
\end{align}
for a constant $C_8$.

Proposition \ref{pro3} and \eqref{ali7} imply that
\[\sup_{K}|\xi\cdot m_{\xi,p}(\cdot,w_1)-\xi\cdot m_{\xi,p}(\cdot,w_2)|\le C_8'|w_1-w_2|\]
for a constant $C_8'$. Theorem \ref{thm4} is proved. 
\end{proof}

At last, we prove Theorem \ref{thm5}.

\begin{proof}[Proof of Theorem \ref{thm5}.]
Let $U_2\Subset U_1\Subset\Omega$. Then 
\[h\coloneqq \frac{\xi\cdot K_{\xi,p}(\cdot,z)-\xi\cdot K_{\xi,p}(\cdot,w)}{z-w}\]
is a uniformly bounded holomorphic function on $U_1$ for every $z\neq w\in U_2$ by Theorem \ref{thm3}. Cauchy's estimate implies that
\[\sup_{U_2}\left|\frac{\partial h}{\partial x_j}\right|\le C_1\]
for a constant $C_1$. Note that
\[\int_{\Omega}|K_{\xi,p}(\cdot,z)|^p=|K_{\xi,p}(z)|^{p-1}\]
is uniformly bounded for any $z\in U_1$ by \eqref{align4}. It follows from Proposition \ref{pro3} and Cauchy's estimate that
\[\Bigg|\frac{\partial (\xi\cdot K_{\xi,p}(\cdot,z))}{\partial x_j}\bigg|_z-\frac{\partial (\xi\cdot K_{\xi,p}(\cdot,z))}{\partial x_j}\bigg|_w \Bigg| \le C_2|z-w|\]
for a constant $C_2$ and any $z,w\in U_2$. As a result, we get
\begin{equation}\label{ali8}
    \begin{aligned} 
& \Bigg|\frac{\partial (\xi\cdot K_{\xi,p}(\cdot,z))}{\partial x_j}\Big|_z-\frac{\partial (\xi\cdot K_{\xi,p}(\cdot,w))}{\partial x_j}\Big|_w\Bigg|\\
\le&\  \Bigg|\frac{\partial (\xi\cdot K_{\xi,p}(\cdot,z))}{\partial x_j}\Big|_z-\frac{\partial (\xi\cdot K_{\xi,p}(\cdot,z))}{\partial x_j}\Big|_w\Bigg|+\left|\frac{\partial h}{\partial x_j}(w)\right||z-w|\\
\le&\  (C_1+C_2)|z-w|
    \end{aligned}
\end{equation}
for any $z,w\in U_2$.

Let $e_1,\ldots,e_{2n}$ be the standard basis in $\mathbb{R}^{2n}=\mathbb{C}^n$. Then
\begin{align}
&K_{\xi,p}(z+te_j)^{-1}-K_{\xi,p}(z)^{-1}\notag\\
=&\ m_{\xi,p}(z+te_j)^p-m_{\xi,p}(z)^p\notag\\
\le&\  p\re \int_{\Omega}|m_{\xi,p}(\cdot,z+te_j)|^{p-2}\overline{m_{\xi,p}(\cdot,z+te_j)}(m_{\xi,p}(\cdot,z+te_j)-m_{\xi,p}(\cdot,z))\notag
\end{align}
by \eqref{al3}. Note that $\xi\cdot m_{\xi,p}(z,z)=\xi\cdot m_{\xi,p}(z+te_j,z+te_j)=1$. We get from Theorem \ref{thm1}, inequality \eqref{align3} and Proposition \ref{pro4} that
\begin{align}
&K_{\xi,p}(z+te_j)^{-1}-K_{\xi,p}(z)^{-1}\notag\\
\le&\ \frac{p}{K_{\xi,p}(z+te_j)}\re \{\xi\cdot m_{\xi,p}(z+te_j,z+te_j)-\xi\cdot m_{\xi,p}(z+te_j,z)\}\notag\\
=&\ -\frac{p}{K_{\xi,p}(z+te_j)}\re \{\xi\cdot m_{\xi,p}(z+te_j,z)-\xi\cdot m_{\xi,p}(z,z)\}\notag\\
\le&\ -\frac{pt}{K_{\xi,p}(z)}\re \frac{\partial (\xi\cdot m_{\xi,p}(\cdot,z))}{\partial x_j}\bigg|_z+C_3|t|^2\notag
\end{align}
for a constant $C_3$  as $t$ tends to $0$.

Similarly, we have
\begin{align}
&K_{\xi,p}(z)^{-1}-K_{\xi,p}(z+te_j)^{-1}\notag\\
\le&\  \frac{p}{K_{\xi,p}(z)}\cdot\re \{\xi\cdot m_{\xi,p}(z,z)-\xi\cdot m_{\xi,p}(z,z+te_j)\}\notag\\
=&\ -\frac{p}{K_{\xi,p}(z+te_j)}\cdot\re \{\xi\cdot m_{\xi,p}(z,z+te_j)-\xi\cdot m_{\xi,p}(z+te_j,z+te_j)\}\notag\\
\le&\  \frac{pt}{K_{\xi,p}(z)}\cdot\re \frac{\partial (\xi\cdot m_{\xi,p}(\cdot,z+te_j))}{\partial x_j}\bigg|_z+C_4|t|^2\notag
\end{align}
for a constant $C_4$  as $t$ tends to $0$. Combining the above results with \eqref{ali8}, we conclude that
\[\frac{\partial K_{\xi,p}^{-1}}{\partial x_j}(z)=-\frac{p}{K_{\xi,p}(z)}\cdot\re \frac{\partial (\xi\cdot m_{\xi,p}(\cdot,z))}{\partial x_j}\bigg|_z,\]
which implies that
\[\frac{\partial K_{\xi,p}}{\partial x_j}(z)=-K_{\xi,p}(z)^2\frac{\partial K_{\xi,p}^{-1}}{\partial x_j}(z)=p\re \frac{\partial (\xi\cdot K_{\xi,p}(\cdot,z))}{\partial x_j}\bigg|_z.\]
Combining with \eqref{ali8}, we get $K_{\xi,p}\in C_{\loc}^{1,1}(\Omega)$.
\end{proof}

\section{Plurisubharmonicity and boundary behavior}
In this section, we consider some other properties of the $p$-Bergman kernel with respect to $\xi$, mainly about the plurisubharmonic properties and boundary behavior.

We still fix a $\xi\in\ell_1^{(n)}$ satisfying that there exists an $\alpha_0\in\mathbb{N}^n$ such that $\xi_{\alpha_0}\neq 0$, and $\Omega\Subset\mathbb{C}^n$ denotes a bounded domain. 

Proposition \ref{pro2} shows that
\[\log K_{\xi,p}(z)=\sup\big\{p\log |(\xi\cdot f)(z)| \colon \|f\|_p=1\big\}.\]
Then it can be deduced from Lemma \ref{lemma2} and Proposition \ref{pro4} that
$\log K_{\xi,p}(z)$ is a plurisubharmonic function on $\Omega$. Moreover,
\begin{Proposition}
\label{pro8}
Suppose that $p\ge2$ and $\xi_{\alpha}=0$ for all $\alpha>\alpha_0$ satisfying  $|\alpha|=|\alpha_0|+1$, then $\log K_{\xi,p}(z)$ is a strictly plurisubharmonic function.
\end{Proposition}
\begin{proof}
It suffices to prove that for any $z\in\Omega$ , there exists a positive constant $c$ such that
\[\liminf_{r\rightarrow 0+}\frac{1}{r^2} \left(\frac{1}{2\pi}\int_0^{2\pi}\log K_{\xi,p}(z+re^{i\theta }X)d\theta-\log K_{\xi,p}(z)\right)>c|X|^2\]
for any $X\in \mathbb{R}^n$. By a translation, we can assume that $z=o$. 

Arbitrarily choose $h\in A^p(\Omega)$ satisfying $(\xi\cdot h)(o)=0$. Set $h_t\coloneqq m_{\xi,p}(\cdot,o)+th$, where $t\in\mathbb{C}$. Proof of Lemma \ref{lemma5} implies that 
\[\frac{\partial \|h_t\|_p^p}{\partial t}(0)=\frac{\partial \|h_t\|_p^p}{\partial \bar{t}}(0)=0.\]
Since 
\begin{align}
&\frac{\partial^2 |h_t|^p}{\partial t^2}=\frac{p(p-2)}{4}|h_t|^{p-4}(\bar{h}_th)^2,\notag\\
&\frac{\partial^2 |h_t|^p}{\partial t\partial \bar{t}}=\frac{p^2}{4}|h_t|^{p-2}|h|^2\notag
\end{align}
by direct calculations,
\begin{align}
&\left|\frac{\partial^2 |h_t|^p}{\partial t^2}\right|\le\frac{p(p-2)}{4}(|m_{\xi,p}(\cdot,o)|+|h|)^{p-2}|h|^2,\notag\\
&\left|\frac{\partial^2 |h_t|^p}{\partial t\partial \bar{t}}\right|=\frac{p^2}{4}(|m_{\xi,p}(\cdot,o)|+|h|)^{p-2}|h|^2\notag
\end{align}
for $|t|\le 1$. H\"older's inequality implies that  
\[\int_{\Omega}(|m_{\xi,p}(\cdot,o)|+|h|)^{p-2}|h|^2\le\|h\|_p^2\left\{\int_{\Omega}(|m_{\xi,p}(\cdot,o)|+|h|)^p\right\}^{1-\frac{2}{p}}\eqqcolon c<\infty.\]
Therefore the dominated convergence theorem shows that 
\begin{align}
\label{alig2}
\left|\frac{\partial^2 \|h_t\|_p^p}{\partial t^2}(0)\right|\le \frac{p(p-2)}{4}c, \quad \left|\frac{\partial^2 \|h_t\|_p^p}{\partial t\partial \bar{t}}(0)\right|\le \frac{p^2}{4}c.
\end{align}

Denote the coordinate by $z\coloneqq (z_1,\ldots,z_n)$. For any $0\neq X\in\mathbb{R}^n$, set
\[h^X\coloneqq \left(\sum_{j=1}^n\frac{X_jz_j}{((\alpha_0)_j+1)\|X\|}\right)z^{\alpha_0}.\]
By the assumption that $\xi_{\alpha}=0$ for all $\alpha>\alpha_0$ with $|\alpha|=|\alpha_0|+1$,
we see $(\xi\cdot h^X)(o)=0$. Since the $L^p$ norm of $h^X$ is uniformly bounded by a constant independent of $X$, it follows from \eqref{alig2} that
\begin{align}
\label{alig3}
\|h^X_t\|_p^p\le \|m_{\xi,p}(\cdot,o)\|_p^p\big(1+C_1|t|^2+o(|t|^2)\big)
\end{align}
for a constant $C_1>0$.

At $re^{i\theta}X=(re^{i\theta}X_1,\ldots,re^{i\theta}X_n)$, 
\begin{equation*}
    \begin{aligned}
        h^X=re^{i\theta}\|X\|(z-re^{i\theta}X)^{\alpha_0}&+\left(\sum_{j=1}^n\frac{X_j(z_j-re^{i\theta}X_j)}{((\alpha_0)_j+1)\|X\|}\right)(z-re^{i\theta}X)^{\alpha_0}\\
        &+\sum_{\alpha<\alpha_0}a_{\alpha}(z-re^{i\theta}X)^{\alpha},
    \end{aligned}
\end{equation*}
where $a_{\alpha}=O(r^2)$ for any $\alpha<\alpha_0$. Therefore 
\[(\xi\cdot h^X)(re^{i\theta}X)=re^{i\theta}\|X\|\xi_{\alpha_0}+O(r^2).\]

Since $\xi\cdot m_{\xi,p}(\cdot,o)$ is a holomorphic function satisfying that $(\xi\cdot m_{\xi,p}(\cdot,o))(o)=1$, $(\xi\cdot m_{\xi,p}(\cdot,o))(re^{i\theta}X)=1+O(r)$.

It then follows that
\[(\xi\cdot h^X_t)(re^{i\theta}X)=\left\{(\xi\cdot m_{\xi,p}(\cdot,o))(re^{i\theta}X)\right\}(1+tre^{i\theta}\|X\|\xi_{\alpha_0}+O(|t|r^2)),\]
which implies that
\begin{align}
\label{alig4}
|(\xi\cdot h^X_t)(re^{i\theta}X)|^p=\left|(\xi\cdot m_{\xi,p}(\cdot,o))(re^{i\theta}X)\right|^p|1+ptre^{i\theta}\|X\|\xi_{\alpha_0}+O(|t|r^2)|.
\end{align}

It follows from Proposition \ref{pro2}, inequality \eqref{alig3} and equality \eqref{alig4} that
\begin{align}
\label{alig5}
K_{\xi,p}(re^{i\theta}X)\ge\frac{\left|(\xi\cdot m_{\xi,p}(\cdot,o))(re^{i\theta}X)\right|^p}{\|m_{\xi,p}(\cdot,o)\|_p^p}\cdot\frac{|1+ptre^{i\theta}\|X\|\xi_{\alpha_0}+O(|t|r^2)|}{1+C_1|t|^2+o(|t|^2)}.
\end{align}

Let $t=\frac{pr\|X\|\bar{\xi}_{\alpha_0}}{2C_1}e^{-i\theta}$. Then inequality \eqref{alig5} reduces to

\[K_{\xi,p}(re^{i\theta }X)\ge \frac{\left|(\xi\cdot m_{\xi,p}(\cdot,o))(re^{i\theta}X)\right|^p}{\|m_{\xi,p}(\cdot,o)\|_p^p}\cdot\frac{1+\frac{p^2r^2\|X\|^2|\xi_{\alpha_0}|^2}{2C_1}+O(|r|^3)}{1+\frac{p^2r^2\|X\|^2|\xi_{\alpha_0}|^2}{4C_1}+o(|r|^2)}.\]
Consequently, 
\begin{align*}
&\liminf_{r\rightarrow 0+}\frac{1}{r^2} \left(\frac{1}{2\pi}\int_0^{2\pi}\log K_{\xi,p}(re^{i\theta }X)d\theta-\log K_{\xi,p}(o)\right)\\
\ge&\  \liminf_{r\rightarrow 0+}\frac{1}{r^2} \left(\frac{1}{2\pi}\int_0^{2\pi}\log \frac{\left|(\xi\cdot m_{\xi,p}(\cdot,o))(re^{i\theta}X)\right|^p}{\|m_{\xi,p}(\cdot,o)\|_p^p}d\theta-\log K_{\xi,p}(o)\right)+\frac{p^2\|X\|^2|\xi_{\alpha_0}|^2}{4C_1}\\
\ge&\  \frac{p^2\|X\|^2|\xi_{\alpha_0}|^2}{4C_1},
\end{align*}
where the second inequality follows from the plurisubharmonic property of 
\[\log \frac{\left|(\xi\cdot m_{\xi,p}(\cdot,o))\right|^p}{\|m_{\xi,p}(\cdot,o)\|_p^p}.\]

The proof of Proposition \ref{pro8} is done.
\end{proof}

More accurate results about the strictly plurisubharmonic property of $p$-Bergman kernels can be referred to \cite{CZ12}.

It is known that for a complex subvariety $S$ of $\Omega$, any $f\in A^2(\Omega\backslash S)$ can be extended to $\Omega$, which implies that $K_{\Omega\backslash S}(z)=K_{\Omega}(z)$ for $z\in\Omega\backslash S$. 
In particularly, $K_{\Omega\backslash S}(z)$ may not be exhaustive when $\Omega$ is a pseudoconvex domain.
On the other hand, Ning--Zhang--Zhou proved in \cite{nzz} that $K_{\Omega,p}(z)$ is exhaustive for $0<p<2$ when $\Omega$ is a pseudoconvex domain.

In order to prove Theorem \ref{thm2}, we need to recall the $L^p$ extension theorem for $p\in (0,2]$.
\begin{Lemma}[see \cite{berndtsson3,GZ-L2ext}]
\label{lemma6}
Let $\Omega\subset\mathbb{C}^n$ be a bounded pseudoconvex, $L\subset\mathbb{C}^n$ be a complex line such that $\Omega\cap L\neq\emptyset$, and $0<p\le 2$. For any $f\in A^p(\Omega\cap L)$, there exists $F\in A^p(\Omega)$ satisfying that $F|_{\Omega\cap L}=f$ and
\[\int_{\Omega}|F|^p\le C \int_{\Omega\cap L}|f|^p,\]
where $C$ is a constant depending only on $n$ and the diameter of $\Omega$. 
\end{Lemma}

Now we prove Theorem \ref{thm2} using the $L^p$ extension theorem.

\begin{proof}[Proof of Theorem \ref{thm2}]
Let $L\subset\mathbb{C}^n$ be a complex line such that $\Omega\cap L\neq\emptyset$. We can assume that
$L=\{z_2=\cdots=z_n=0\}$ by a unitary transformation.

For $z^0=(z_1^0,0,\ldots,0)\in\partial\Omega\cap L$, since the $L^p(\Omega\cap L)$ norm of the function $\frac{1}{z_1-z^0_1}$ is bounded by a constant $c_1$, it follows from Lemma \ref{lemma6} that there exists a holomorphic function $F$ on $\Omega$ whose $L^p(\Omega)$ norm is bounded by a constant $c_2$ such that $F|_{\Omega\cap L}=\frac{1}{z_1-z^0_1}$.

Choose any $z'=(z'_1,0,\ldots,0)\in \Omega\cap L$. Since $(z-z')^{\alpha_0}$ is uniformly bounded on $\Omega$, it can be obtained that $G\coloneqq F\cdot(z-z')^{\alpha_0}$ is a holomorphic function on $\Omega$ whose $L^p(\Omega)$ norm is bounded by a constant $c_3$.

Note that $(\partial^{\alpha_0}G)|_{z'}=\alpha_0 !\cdot F(z')$ and $(\partial^{\alpha}G)|_{z'}=0$ for any $\alpha\not \ge\alpha_0$, which implies that
\[(\xi\cdot G)(z')=F(z')=\frac{1}{z'_1-z_1^0}.\]
It follows from Proposition \ref{pro2} that
\[K_{\xi,p}(z')\ge \frac{1}{c_3^p \cdot |z'-z_0|^p}.\]
Since $L$ and $z'\in L$ are arbitrarily chosen, the proof of Theorem \ref{thm2} is done.
\end{proof}

It follows from the Ohsawa--Takegoshi $L^2$ extension theorem that 
\begin{align}
\label{al11}
K_{2}(z)\ge \frac{c'}{\delta(z)^2}
\end{align}
for a constant $c'$ when $\Omega$ is a bounded pseudoconvex domain with $C^2$ boundary (see \cite{ohsawa2}).
\begin{Proposition}
\label{pro9}
Let $\Omega$ be a bounded pseudoconvex domain with $C^2$ boundary in $\mathbb{C}^n$, and $\xi\in \ell_1^{(n)}$ satisfying $\xi_{\alpha_0}\neq 0$ for some $\alpha_{0}\in \mathbb{N}^{n}$ while $\xi_{\alpha}=0$ for any $\alpha> \alpha_0$. Then
\begin{align}
K_{\xi,2}(z)\ge \frac{c}{\delta(z)^2}\notag
\end{align}
for a constant $c$.
\end{Proposition}
\begin{proof}
By \eqref{al11}, there exists a constant $C_1$ such that for any $z_0\in\Omega$, there exists $F\in A^2(\Omega)$ satisfying that
\[
\int_{\Omega}|F|^2\le C_1\delta(z_0)^2|F(z_0)|^2 .
\]
Since $(z-z_0)^{\alpha_0}$ is uniformly bounded in $\Omega$, if we let $G=F\cdot (z-z_0)^{\alpha_0}$, then
\begin{equation}
\label{al12}
\int_{\Omega}|G|^2\le C_2\delta(z_0)^2|F(z_0)|^2 
\end{equation}
for a constant $C_2$. As $(\partial^{\alpha_0}G)|_{z_0}=\alpha_0 !\cdot F(z_0)$ and $(\partial^{\alpha}G)|_{z_0}=0$ for any $\alpha\not \ge\alpha_0$, we have $(\xi\cdot G)(z_0)=F(z_0)$. Then it follows from Proposition \ref{pro2} and \eqref{al12} that
\[K_{\xi,2}(z_0)\ge\frac{|(\xi\cdot(F\cdot(z-z_0)^{\alpha_0}))(z_0)|^2}{\int_{\Omega}|F\cdot (z-z_0)^{\alpha_0}|^2}\ge \frac{1}{C_2\delta(z_0)^2}.\]
Proposition \ref{pro9} is proved.
\end{proof}

\section{Higher order Bergman kernel and $\xi$-Bergman kernel}
In this section, we give the relations between the $L^p$ versions of higher order Bergman kernels and $\xi$-Bergman kernels. We prove Theorem \ref{thm-inf} for all $p\in [1,+\infty)$ by the Hahn--Banach Theorem, and give an alternative proof for the case $p=2$.

\subsection{Proof of Theorem \ref{thm-inf}}
In the following, we give the proof of Theorem \ref{thm-inf}. Indeed, Theorem \ref{thm-inf} is a special case of the following theorem.

Let $\mathbb{N}^n=\mathbf{E}_0\sqcup\mathbf{E}_1$ be a disjoint partition and $\eta=(\eta_{\alpha})\in\ell_1^{(n)}$ satisfying that $\mathbf{E}_0$ is a finite set and $\eta_{\alpha}=0$ for any $\alpha\in\mathbf{E}_0$.

For any $\xi=(\xi_{\alpha})\in\ell^{(n)}_1$, denote $\xi\in \mathbf{S}_{\eta}$ if $\xi_{\alpha}=\eta_{\alpha}$ for all $\alpha\in\mathbf{E}_1$. 

Let $D$ be a domain in $\mathbb{C}^n$, $z\in D$, and $p\in [1,+\infty]$. Denote
\begin{flalign*}
    \begin{split}
        \mathscr{A}^p_{\mathbf{E}_0}(D)&\coloneqq \left\{f\in A^p(D) \colon D^{\alpha}f(z)=0, \ \forall \alpha\in\mathbf{E}_0\right\},\\
        \kappa_{\eta,D}^{\mathbf{E}_0,p}(z)&\coloneqq \sup\left\{|(\eta\cdot f)(z)| \colon f\in \mathscr{A}_{\mathbf{E}_0}^p(D), \ \|f\|_{p,D}\le 1\right\},\\
        \kappa^p_{\xi,D}(z)&\coloneqq \left\{|(\xi\cdot f)(z) \colon f\in A^p(\Omega), \ \|f\|_{p,D}\le 1\right\}, \quad \forall\, \xi\in\ell^{(n)}_1,
    \end{split}
\end{flalign*}
where $A^{\infty}(D)\coloneqq L^{\infty}(D)\cap\mathcal{O}(D)$, and $\|f\|_{\infty, D}\coloneqq \sup_{w\in D}|f(w)|$. 

\begin{Theorem}
   The following equality holds for any $p\in [1,+\infty]$,
\[\kappa_{\eta,D}^{\mathbf{E}_0,p}(z)=\inf_{\xi\in\mathbf{S}_{\eta}}\kappa_{\xi,D}^p(z)=\min_{\xi\in\mathbf{S}_{\eta}}\kappa_{\xi,D}^p(z).\]
\end{Theorem}

\begin{proof}
    It is clear that $\kappa_{\eta,D}^{\mathbf{E}_0,p}(z)\le \kappa_{\xi,D}^p(z)$ for any $\xi\in\mathbf{S}_{\eta}$. Then we may assume $\kappa_{\eta,D}^{\mathbf{E}_0,p}(z)<+\infty$.   
    
    We prove the theorem in three steps.

    \textbf{Step 1}: Suppose $D$ is bounded, and we prove $\kappa_{\eta,D}^{\mathbf{E}_0,p}(z)=\min_{\xi\in\mathbf{S}_{\eta}}\kappa_{\xi,D}^p(z)$.

    In case $D$ is bounded, any polynomial on $\mathbb{C}^n$ belongs to $A^p(D)$. Then one can see $\mathscr{A}^p_{\mathbf{E}_0}(D)$ is a non-empty closed subspace of the Banach space $A^p(D)$, and for any $f\in\mathscr{A}^p_{\mathbf{E}_0}(D)$,
    \[|(\eta\cdot f)(z)|\le \kappa_{\eta,D}^{\mathbf{E}_0,p}(z)\cdot \|f\|_{p,D}.\]
    By the Hahn--Banach Theorem, the continuous linear functional $\Lambda_{\eta}\colon f\mapsto (\eta\cdot f)(z)$ on $\mathscr{A}^p_{\mathbf{E}_0}(D)$ can be extended to a continuous linear functional $\Lambda$ on $A^p(D)$, with $\Lambda|_{\mathscr{A}^p_{\mathbf{E}_0}(D)}=\Lambda_{\eta}$, and
    \[|\Lambda(f)|\le \kappa_{\eta,D}^{\mathbf{E}_0,p}(z)\cdot \|f\|_{p,D}, \quad \forall\, f\in A^p(D).\]

    Next, we prove that there exists some $\xi\in\mathbf{S}_{\eta}$ such that $\Lambda_{\xi}=\Lambda$, where $\Lambda_{\xi}\colon f\mapsto (\xi\cdot f)(z)$. This will imply $\xi\in\mathbf{S}_{\eta}$ and $\|\Lambda_{\xi}\|\le \kappa_{\eta,D}^{\mathbf{E}_0,p}(z)$, which yields $\kappa^p_{\xi,D}(z)=\kappa_{\eta,D}^{\mathbf{E}_0,p}(z)$. Let $f\in A^p(D)$. By Taylor's expansion, one can uniquely decompose $f$ into $f=f_0+f_1$, where
    \begin{flalign*}
        \begin{split}
            \left\{\alpha \colon D^{\alpha}f_0(z)\neq 0\right\}&\subset \mathbf{E}_0, \\
             \left\{\alpha \colon D^{\alpha}f_1(z)\neq 0\right\}&\subset\mathbf{E}_1.
        \end{split}
    \end{flalign*}
    Since $\mathbf{E}_0$ is a finite set, $f_0$ is a polynomial, which belongs to $A^p(D)$. Then $f_1\in \mathscr{A}^p_{\mathbf{E}_0}(D)$, and $\Lambda(f_1)=\Lambda_{\eta}(f_1)$. We get
    \[\Lambda(f)=\Lambda(f_0)+\Lambda(f_1)=\Lambda(f_0)+\Lambda_{\eta}(f_1),\]
    where
    \[\Lambda(f_0)=\sum_{\alpha\in\mathbf{E}_0}\Lambda(z^{\alpha})\cdot\frac{D^{\alpha}f(z)}{\alpha!}, \quad \forall\, f\in A^p(D).\]
    Let $\xi=(\xi_{\alpha})\in\ell_1^{(n)}$, where
    \begin{equation*}
        \xi_{\alpha}=
            \begin{cases}
                \Lambda(z^{\alpha}), \ & \alpha\in\mathbf{E}_0,\\
                \eta_{\alpha}, \ & \alpha\in\mathbf{E}_1.\\
            \end{cases}
    \end{equation*}
    Then $\xi\in\mathbf{S}_{\eta}$ satisfying $\Lambda_{\xi}=\Lambda$, which implies that $\kappa_{\eta,D}^{\mathbf{E}_0,p}(z)=\min_{\xi\in\mathbf{S}_{\eta}}\kappa_{\xi,D}^p(z)$.

    \textbf{Step 2}: For a general domain $D$ in $\mathbb{C}^n$, we prove $\kappa_{\eta,D}^{\mathbf{E}_0,p}(z)=\inf_{\xi\in\mathbf{S}_{\eta}}\kappa_{\xi,D}^p(z)$.

    Let $D_1\subset D_2\subset\cdots\subset D_j\subset\cdots$ be an increasing sequence of bounded domains in $\mathbb{C}^n$ such that $D=\bigcup_{j=1}^{\infty}D_j$ and $z\in D_1$. By \textbf{Step 1},
    \[\kappa_{\eta,D_j}^{\mathbf{E}_0,p}(z)=\min_{\xi\in\mathbf{S}_{\eta}}\kappa_{\xi,D_j}^p(z), \quad \forall\, j\ge 1.\]
    By Montel's Theorem, we also have (e.g. see the proof of Proposition \ref{prop-exhaustion})
\[\lim_{j\to \infty}\kappa_{\eta,D_j}^{\mathbf{E}_0,p}(z)=\kappa_{\eta,D}^{\mathbf{E}_0,p}(z),\]
and
\[\lim_{j\to \infty}\kappa_{\xi,D_j}^p(z)=\kappa_{\xi,D_j}^p(z), \quad \forall\, \xi\in\mathbf{S}_{\eta}.\]  
    Consequently, $\kappa_{\eta,D}^{\mathbf{E}_0,p}(z)=\inf_{\xi\in\mathbf{S}_{\eta}}\kappa_{\xi,D}^p(z)$.

    \textbf{Step 3}: We prove $\inf_{\xi\in\mathbf{S}_{\eta}}\kappa_{\xi,D}^p(z)=\min_{\xi\in\mathbf{S}_{\eta}}\kappa_{\xi,D}^p(z)$.

    Note that $\mathbf{S}_{\eta}$ is a linear space of finite dimension ($\dim\mathbf{S}_{\eta}=|\mathbf{E}_0|$). Proposition \ref{pro4} shows that the function $\xi\mapsto \kappa_{\xi,D}^p(z)$ is a continuous function on $\mathbf{S}_{\eta}$. For any $\alpha\in\mathbf{E}_0$, we may assume
    \[\big\{f\in A^p(D) \colon D^{\alpha}f(z)\neq 0\big\}\neq\emptyset.\]
    Otherwise, the value of $\xi_{\alpha}$ ($\alpha$-th position of $\xi$) does not affect $\kappa^p_{\xi,D}(z)$. Then now, one can find that
    \[\lim_{|\xi_{\alpha}|\to+\infty}\kappa_{\xi,D}^p(z)=+\infty,\]
    where the limit is taken when the values of the other positions of $\xi$ are fixed except the $\alpha$-th position. It follows that $\inf_{\xi\in\mathbf{S}_{\eta}}\kappa_{\xi,D}^p(z)=\inf_{\xi\in \mathscr{K}}\kappa_{\xi,D}^p(z)$, where $\mathscr{K}$ is a compact subset of $\mathbf{S}_{\eta}$. Thus, the continuity of $\kappa_{\xi,D}^p(z)$ in $\xi$ implies that $\inf_{\xi\in\mathbf{S}_{\eta}}\kappa_{\xi,D}^p(z)=\min_{\xi\in\mathbf{S}_{\eta}}\kappa_{\xi,D}^p(z)$ holds.
    
    The proof is complete.
\end{proof}

\subsection{An alternative proof of Theorem \ref{thm-inf} for $p=2$}
Before the proof, we recall a construction of the complete orthonormal basis of Bergman spaces. For any two multi-indices $\alpha,\beta\in\mathbb{N}^n$, where $\alpha=(\alpha_1,\ldots,\alpha_n)$, $\beta=(\beta_1,\ldots,\beta_n)$, we write $\alpha\prec\beta$, if

(1) $|\alpha|<|\beta|$; or

(2) $|\alpha|=|\beta|$, and there exists $j\in \{1,\ldots,n-1\}$ such that $\alpha_n=\beta_n$, $\ldots$, $\alpha_{j+1}=\beta_{j+1}$, $\alpha_{j}<\beta_{j}$. 

\begin{Lemma}[cf. {\cite[Appendix]{BG3}}]\label{lem-basis}
    Let $\Omega\subset\mathbb{C}^n$ be a domain, and fix $z\in \Omega$. There exists a subset $\mathbf{E}$ of $\mathbb{N}^n$, and a complete orthonormal basis $\{\sigma_{\alpha}\}_{\alpha\in\mathbf{E}}$ of $A^2(\Omega)$, such that 
    \begin{enumerate}
    \item $\min\{\gamma\in\mathbb{N}^n \colon (D^{\gamma}\sigma_{\alpha})(z_0)\neq 0\}=\alpha$ for every $\alpha\in \mathbf{E}$, where the minimum is w.r.t. `$\prec$';
    \item$(D^{\beta}\sigma_{\alpha})(z_0)=0$, for every $\alpha \in\mathbf{E}$ and every $\beta\in\mathbb{N}^n$ with $\beta\prec\alpha$.
\end{enumerate} 
\end{Lemma}

In the following, we give an alternative proof of Theorem \ref{thm-inf} for $p=2$, which provides an approach to computing the minimizing functional $\xi$.

\begin{proof}[An alternative proof of Theorem \ref{thm-inf} for $p=2$]
We may assume $\inf_{\xi\in\mathbf{S}_H}K_{\xi,\Omega}(z)>0$, otherwise the result is trivial. According to (\ref{eq-KHKxi}), we only need to prove that there exists some $\xi_H\in\mathbf{S}_H$ such that $K^H_{\Omega}(z)=K_{\xi_H,\Omega}(z)$.

Using Lemma \ref{lem-basis}, we can get a complete orthonormal basis $\{\sigma_{\alpha}\}_{\alpha\in\mathbf{E}}$ of $A^2(\Omega)$ satisfying the conditions in Lemma \ref{lem-basis}. For any $f\in A^2(\Omega)$, if we write
\begin{equation*}
    f=\sum_{\alpha\in\mathbf{E}}a_{\alpha}\sigma_{\alpha},
\end{equation*}
where $a_{\alpha}\in\mathbb{C}$, then 
\begin{equation}\label{eq-Dbetafz}
        (D^{\beta}f)(z)=\sum_{\alpha\in\mathbf{E}}a_{\alpha}(D^{\beta}\sigma_{\alpha})(z)=\sum_{\alpha\preceq\beta,\, \alpha\in\mathbf{E}}a_{\alpha}(D^{\beta}\sigma_{\alpha})(z).
\end{equation}
Now for any $\xi=(\xi_{\beta})\in\mathbf{S}_H$, we can write
\begin{equation*}
    (\xi\cdot f)(z)=\sum_{|\beta|\le k}\frac{\xi_{\beta}}{\beta!}(D^{\beta}f)(z).
\end{equation*}
Combining with (\ref{eq-Dbetafz}), we have
\begin{align*}
    \begin{split}
        (\xi\cdot f)(z)&=\sum_{|\beta|\le k}\frac{\xi_{\beta}}{\beta!}\left(\sum_{\alpha\preceq\beta,\, \alpha\in\mathbf{E}}a_{\alpha}(D^{\beta}\sigma_{\alpha})(z)\right)\\
        &=\sum_{|\alpha|\le k,\, \alpha\in\mathbf{E}}\left(\sum_{\beta\succeq\alpha, |\beta|\le k}\xi_{\beta}\frac{(D^{\beta}\sigma_{\alpha})(z)}{\beta!}\right)a_{\alpha}.
    \end{split}
\end{align*}
Set
\begin{equation}\label{calpha}
    c_{\alpha}\coloneqq \sum_{\beta\succeq\alpha,\, |\beta|\le k}\xi_{\beta}\frac{(D^{\beta}\sigma_{\alpha})(z)}{\beta!}
\end{equation}
for each $\alpha\in\mathbf{E}$ with $|\alpha|\le k$. Then using the Cauchy--Schwarz inequality, we can compute that
\begin{align*}
    \begin{split}
        K_{\xi,\Omega}(z)&\coloneqq\sup_{f\in A^2(\Omega)}\frac{|(\xi\cdot f)(z)|^2}{\int_{\Omega}|f|^2}\\
        &=\sup_{\{a_{\alpha}\}_{\alpha\in\mathbf{E}}\in l^2}\frac{\left|\sum_{|\alpha|\le k, \alpha\in\mathbf{E}}c_{\alpha}a_{\alpha}\right|^2}{\sum_{\alpha\in\mathbf{E}}|a_{\alpha}|^2}\\
        &=\sum_{|\alpha|\le k,\, \alpha\in\mathbf{E}}|c_{\alpha}|^2,
    \end{split}
\end{align*}
where the supremum is achieved if and only if
\begin{equation*}
    f=C\sum_{\alpha\le k,\, \alpha\in\mathbf{E}}\bar{c}_{\alpha}\sigma_{\alpha},
\end{equation*}
for some $C\in\mathbb{C}\setminus\{0\}$. Note that
\begin{equation*}
    (D^{\beta}\sigma_{\alpha})(z)=0, \quad \forall\, \alpha\in\mathbf{E}, \ \beta\in\mathbb{N}^n, \ \beta\prec\alpha.
\end{equation*}
If we let
\begin{equation}\label{eq-calpha}
    c_{\alpha}=0, \quad \forall\, \alpha\in\mathbf{E}, \ |\alpha|<k,
\end{equation}
then
\begin{equation*}
    K_{\xi,\Omega}(z)=\frac{|(\xi\cdot f_0)(z)|^2}{\int_{\Omega}|f_0|^2},
\end{equation*}
for some $f_0\in A^2(\Omega)$ with $\|f_0\|_{L^2(\Omega)}=1$ and $f_0^{(j)}(z)=0$, $j=0,\ldots,k-1$. Since $\xi\in\mathbf{S}_H$, we also have
\begin{equation*}
    (\xi\cdot f_0)(z)=P_H(f_0)(z).
\end{equation*}
Now it is clear that
\begin{equation*}
    K_{\xi, \Omega}(z)=K^H_{\Omega}(z)
\end{equation*}
for every $\xi=(\xi_{\beta})\in\mathbf{S}_H$ such that (\ref{eq-calpha}) holds. Now we only need to prove the existence of such $\xi$. Actually, the equation (\ref{eq-calpha}) is equivalent to the following linear equation system:
\begin{equation}
    \sum_{\beta\succeq\alpha, \, |\beta|\le k}\xi_{\beta}\frac{(D^{\beta}\sigma_{\alpha})(z)}{\beta!}=0, \quad \forall\, \alpha\in\mathbf{E}, \ |\alpha|\le k-1.
\end{equation}
This linear equation system must have solutions since
\begin{equation*}
    (D^{\alpha}\sigma_{\alpha})(z)\neq 0, \quad \forall\, \alpha\in\mathbf{E}.
\end{equation*}
In summary, we get
\begin{equation*}
    K^H_{\Omega}(z)=\inf_{\xi\in\mathbf{S}_H}K_{\xi,\Omega}(z),
\end{equation*}
and the infimum can be achieved, i.e.,
\begin{equation*}
    K^H_{\Omega}(z)=\inf_{\xi\in\mathbf{S}_H}K_{\xi,\Omega}(z)=\min_{\xi\in\mathbf{S}_H}K_{\xi,\Omega}(z).
\end{equation*} 
\end{proof}

\section{Proofs of Theorem \ref{thm-Kp} and its corollaries}
Before giving the proofs of Theorem \ref{thm-Kp} and its corollaries, we firstly recall the log-plurisubharmonicity of fiberwise $\xi$-Bergman kernels, and demonstrate the $L^p$ version for $p\in (0,2]$.

\subsection{Log-plurisubharmonicity}\label{sec-logpsh}
Let $\Omega\subset\mathbb{C}^{n+1}$ be a pseudoconvex domain, with coordinate $(z,\zeta)$, where $z\in\mathbb{C}^n$, $\zeta\in\mathbb{C}$. Let $\varpi_1$, $\varpi_2$ be the natural projections $\varpi_1(z,\zeta)=\zeta$, $\varpi_2(z,\zeta)=z$ on $\Omega$, $\omega\coloneqq \varpi_1(\Omega)$. Denote by $K_{\xi,\Omega_{\zeta}}(z)$ the Bergman kernels on the domains $\Omega_{\zeta}\coloneqq \varpi_1^{-1}(\zeta)\cap \Omega$ with respect to some fixed $\xi\in \ell_1^{(n)}$.

The following generalization of Berndtsson's log-plurisubharmonicity of fiberwise Bergman kernels was obtained in \cite{BG1}.
\begin{Proposition}[\cite{BG1}]\label{logpsh-xi}
    $\log K_{\xi,\Omega_{\zeta}}(z)$ is a plurisubharmonic function on $\Omega$ with respect to $(z,\zeta)$.
\end{Proposition}

\begin{Remark}
In \cite{BG1}, the authors assume that $\Omega_{\zeta}$ is bounded for any $\zeta\in \omega$, but it is clear that this assumption is not needed (one may modify the original proof, or approximate the unbounded pseudoconvex domain by bounded pseudoconvex domains).

The main method of getting Proposition \ref{logpsh-xi} is using the optimal $L^2$ extension theorem and Guan--Zhou Method.
\end{Remark}

Considering that there is also optimal $L^p$ extension theorem for any $p\in (0,2]$, one can expect to establish the log-plurisubharmonicity for fiberwise $p$-Bergman kernels with respect to some fixed functional $\xi$.

\begin{Lemma}[Optimal $L^p$ extension theorem, \cite{GZ-L2ext}]\label{lem-Lp.ext}
     Let $\Omega\subset\mathbb{C}^{n+1}$ be a pseudoconvex domain, with coordinate $(z,\zeta)$, where $z\in\mathbb{C}^n$, $\zeta\in\mathbb{C}$. Let $\varpi$ be the natural projection $\varpi(z,\zeta)=\zeta$ on $\Omega$, $\Omega'\coloneqq \varpi(\Omega)$. Suppose $\Omega'=\Delta(\zeta_0,r)$ is a disk centered at $\zeta_0\in\mathbb{C}$ with radius $r>0$. Let $p\in (0,2]$. Then for any $f\in A^p(\Omega_{\zeta_0})$, where $\Omega_{\zeta_0}\coloneqq \varpi^{-1}(\zeta_0)\cap\Omega$, there exists $F\in A^p(\Omega)$ such that $F|_{\Omega_{\zeta_0}}=f$, and
     \begin{equation*}
         \frac{1}{\pi r^2}\int_{\Omega}|F|^p\le\int_{\Omega_{\zeta_0}}|f|^p.
     \end{equation*}
\end{Lemma}

Depending on the optimal $L^p$ extension theorem, the following result for fiberwise $p$-Bergman kernels with respect to $\xi$ can be established.

Let $\Omega, \omega, \Omega_{\zeta}$ be as in the settings of Proposition \ref{logpsh-xi}, $\xi\in\ell_1^{(n)}$, and $p\in (0,2]$. Denote by $K_{\xi,\Omega_{\zeta},p}(z)$ the $p$-Bergman kernel with respect to $\xi$ on each fiber $\Omega_{\zeta}$.

\begin{Proposition}\label{prop-p.BK.log.psh}
    $\log K_{\xi,\Omega_{\zeta},p}(z)$ is a plurisubharmonic function on $\Omega$ with respect to $(z,\zeta)$.
\end{Proposition}

\begin{proof}
    The proof is almost as the same as the proof of Proposition \ref{logpsh-xi} in \cite{BG1}, except using the optimal $L^p$ extension theorem rather than the optimal $L^2$ extension theorem.

    Proposition \ref{prop-exhaustion} indicates that $\log K_{\xi,\Omega_{\zeta},p}(z)$ is upper-semicontinuous on $\Omega$, and we have got the plurisubharmonicity of $\log K_{\xi,\Omega_{\zeta},p}(z)$ in $z$ (see Proposition \ref{pro2}). In order to get the mean value inequality of $\log K_{\xi,\Omega_{\zeta},p}(z)$ in $\zeta$, we need to apply the optimal $L^p$ extension theorem and Guan--Zhou method.

    Let $(z_0,\zeta_0)\in\Omega$ and $f\in A^p(\Omega_{\zeta_0})$ such that
    \[K_{\xi,\Omega_{\zeta_0},p}(z_0)=\frac{|(\xi\cdot f)(z_0)|^p}{\int_{\Omega_{\zeta_0}}|f|^p}.\]
    For any $r>0$ with $\Delta(\zeta_0,r)\subseteq\omega$, we prove the sub-mean value inequality of $\log K_{\xi,\Omega_{\zeta},p}(z_0)$ in $\zeta\in\Delta(\zeta_0,r)$. We simply assume $\omega=\Delta(\zeta_0,r)$. Then according to Lemma \ref{lem-Lp.ext}, there exists $F\in A^p(\Omega)$ such that $F|_{\Omega_{\zeta_0}}=f$ and
    \[\frac{1}{\pi r^2}\int_{\Omega}|F|^p\le\int_{\Omega_{\zeta_0}}|f|^p.\]
    Consequently,
    \begin{flalign*}
        \begin{split}
            \log\left(\int_{\Omega_{\zeta_0}}|f|^p\right)&\ge\log\left(\frac{1}{\pi r^2}\int_{\Omega}|F|^p\right)\\
            &= \log\left(\frac{1}{\pi r^2}\int_{\zeta\in\Delta(\zeta_0,r)}\int_{\Omega_{\zeta}}|F_{\zeta}|^p\right)\\
            &\ge \frac{1}{\pi r^2}\int_{\zeta\in\Delta(\zeta_0,r)}\log\left(\int_{\Omega_{\zeta}}|F_{\zeta}|^p\right)\\
            &\ge \frac{1}{\pi r^2}\int_{\zeta\in\Delta(\zeta_0,r)}\bigg(\log|(\xi\cdot F_{\zeta})(z_0)|^p-\log K_{\xi,\Omega_{\zeta},p}(z_0)\bigg),
        \end{split}
    \end{flalign*}
    where $F_{\zeta}(z)\coloneqq F(z,\zeta)$. Since $\xi\in\ell_1^{(n)}$, we can verify that $(\xi\cdot F_{\zeta})(z_0)$ is holomorphic in $\zeta$ (see also \cite{BG1, BG2}), yielding that $\log|(\xi\cdot F_{\zeta})(z_0)|^p$ is subharmonic in $\zeta$. Eventually, we get that
    \[\log K_{\xi,\Omega_{\zeta_0},p}(z_0)\le\frac{1}{\pi r^2}\int_{\zeta\in\Delta(\zeta_0,r)}\log K_{\xi,\Omega_{\zeta},p}(z_0),\]
    which is the sub-mean value inequality of $\log K_{\xi,\Omega_{\zeta},p}(z_0)$ in $\zeta$ on $\Delta(\zeta_0,r)$, and we get the desired result.
\end{proof}

Now we consider a pseudoconvex domain $D\subset\mathbb{C}^n$ containing the origin $o$, with a negative plurisubharmonic function $\varphi$ on $D$ such that $\varphi(o)=-\infty$. Let $\xi\in\ell^{(n)}_1$, and $p\in (0,2]$. Denote by
\[K_{\xi,\{\varphi<t\},p}(o), \quad t\in (-\infty,0]\]
the $p$-Bergman kernels with respect to $\xi$ on the sub-level sets $\{\varphi<t\}$ of $\varphi$. With a similar discussion in \cite{BL16} (see also \cite{BG1, BG2, BG3}), one can obtain from Proposition \ref{prop-p.BK.log.psh} that

\begin{Proposition}\label{prop-log.convex}
    $\log K_{\xi,\{\varphi<t\},p}(o)$ is convex in $t\in (-\infty,0]$.
\end{Proposition}

\begin{proof}
    Note that the domain
    \[\big\{(z,\zeta)\in D\times\mathbb{C}\colon \varphi(z)-\re\zeta<0\big\}\]
    is pseudoconvex in $\mathbb{C}^{n+1}$. Then Proposition \ref{prop-log.convex} implies that $\log K_{\xi,\{\varphi<\re\zeta\},p}(o)$ is plurisubharmonic in $\zeta$. Thus, since $\log K_{\xi,\{\varphi<\re\zeta\},p}(o)$ only depends on $\re\zeta$, we get that $\log K_{\xi,\{\varphi<t\},p}(o)$ is convex in $t\in (-\infty,0]$. 
\end{proof}

\subsection{Pluricomplex Green function}
For a domain $D$ in $\mathbb{C}^n$ and any $z\in D$, denote
\[\mathcal{L}_z(D)\coloneqq \Big\{\phi\in\mathrm{PSH}^-(D) \colon \limsup_{w\to z}\big(\phi(w)-\log|w-z|\big)<+\infty\Big\}.\]
The \emph{pluricomplex Green function} on $D$ with the pole $z$ is defined by:
\[G_{D}(\zeta,z)\coloneqq \sup\big\{\phi(\zeta) \colon \phi\in \mathcal{L}_z(D)\big\}.\]
It is easy to see that $G_{D}(\zeta,z)\le \log|\zeta-z|+O(1)$ when $\zeta$ is near $z$. Especially, for a bounded hyperconvex domain $D$, it is well-known that $G_{D}(\zeta,z)$ is continuous in $\zeta$ and tends to $0$ when $\zeta\to\partial D$ for fixed $z$.

Let $D$ be a pseudoconvex domain in $\mathbb{C}^n$ with $o\in D$. Recall
\[D_a\coloneqq e^{-a}\{G_{D}(\cdot,o)<a\}, \quad \forall\, a\in (-\infty,0].\]
Let $p\in (0,+\infty)$. For any $\xi=(\xi_{\alpha})\in\ell_0^{(n)}$, where
\[\ell_0^{(n)}\coloneqq \big\{(\eta_{\alpha})_{\alpha\in\mathbb{N}^n}\in\ell_1^{(n)} \colon \exists\, m\in\mathbb{N}, \ \text{s.t.} \ \eta_{\alpha}=0, \ \forall\, \alpha \ \text{with} \ |\alpha|\ge m\big\},\]
denote
\[\deg(\xi)\coloneqq \max\big\{k\in\mathbb{N} \colon \exists\, \alpha\in\mathbb{N}^n \ \text{s.t.} \ |\alpha|=k \ \text{and} \ \xi_{\alpha}\neq 0\big\}.\]

\begin{Lemma}\label{lem-BK.bounded.above}
    For any $\xi\in\ell_0^{(n)}$ with $\deg(\xi)=k$, there exists a positive constant $C$, such that for any $a\le 0$, the following inequality holds,
    \[e^{(2n+pk)a}K_{\xi,\{G_{D}(\cdot,o)<a\},p}(o)\le C.\]
\end{Lemma}

\begin{proof}
    We choose some sufficiently small $r_0>0$ and a real constant $M$ such that
    \[G_{D}(z,o)\le \log|z|+M, \quad \forall\, |z|<r_0.\]
    It follows that
    \[\mathbb{B}^n(o,e^{a-M})\subset \{G_{D}(\cdot,o)<a\}\]
    for sufficiently negative $a$. Thus, according to the estimate (\ref{align5}) on balls, we deduce
    \[K_{\xi,\{G_{D}(\cdot,o)<a\},p}(o)\le K_{\xi,\mathbb{B}^n(o,e^{a-M}),p}(o)\le C e^{-(2n+pk)a},\]
    for some positive constant $C$ independent of $a$.
\end{proof}

\subsection{Proof of Theorem \ref{thm-Kp}}

Now we prove Theorem \ref{thm-Kp}.

\begin{proof}[Proof of Theorem \ref{thm-Kp}]
    First, Proposition \ref{prop-log.convex} verifies the convexity of the function
    \[(-\infty,0] \ni a \to \log K_{\xi,\{G<a\},p}(o).\]
    Here we denote the pluricomplex Green function $G_{\Omega}(\cdot,o)$ by $G$ for short. It follows that $e^{(2n+pk)a}K_{\xi,\{G<a\},p}(o)$ is also convex in $a\in (-\infty,a]$. In addition, since $e^{(2n+pk)a}K_{\xi,\{G<a\},p}(o)$ is bounded above by Lemma \ref{lem-BK.bounded.above}, we obtain that the function $e^{(2n+pk)a}K_{\xi,\{G<a\},p}(o)$ is non-decreasing in $a\in (-\infty,0]$.

    Next, we prove the inequality 
    \[\limsup_{a\to -\infty} K^{H,p}_{\Omega_a}(o)\ge K_{I_{\Omega}(o)}^{H,p}(o).\]
    Like what is done in \cite{BlZw20}, we can approximate the pseudoconvex domain $\Omega$ by bounded hyperconvex domains, which allows us to assume that $\Omega$ is a bounded hyperconvex domain. Then as in \cite{BlZw20} for the case $p=2$, the convergence in the sense of Hausdorff (cf. \cite{Zw00}): $\Omega_a\to I_{\Omega}(o)$ for the bounded hyperconvex domain $\Omega$ implies the desired inequality.

    Now for every $\xi\in\mathbf{S}_H$, we obtain
    \begin{flalign*}
        \begin{split}
            \lim_{a\to -\infty}e^{2n+pk}K_{\xi,\{G<a\},p}(o)&\ge \limsup_{a\to -\infty} e^{2n+pk}K^{H,p}_{\{G<a\}}(o)\\
            &=\limsup_{a\to -\infty} K^{H,p}_{\Omega_a}(o)\ge K_{I_{\Omega}(o)}^{H,p}(o).
        \end{split}
    \end{flalign*}

    The proof is complete.
\end{proof}

\subsection{Proofs of the corollaries}

Now we give the proofs of the corollaries of Theorem \ref{thm-Kp}, including Corollary \ref{cor-BS.nontrivial}, \ref{cor-higher.BK.non.decreasing}, and \ref{cor-BS.infinite.dim}.

\begin{proof}[Proof of Corollary \ref{cor-BS.nontrivial}]
    If the space $A^p(I_{\Omega}(o))$ is nontrivial, then there exists a homogeneous $H$ such that $K^{H,p}_{I_{\Omega}(o)}(o)>0$. Hence, choosing any $\xi\in\mathbf{S}_H$, Theorem \ref{thm-Kp} indicates that $K_{\xi,\Omega,p}(o)>0$, yielding that $A^p(\Omega)$ is nontrivial.
\end{proof}

\begin{proof}[Proof of Corollary \ref{cor-higher.BK.non.decreasing}]
    According to Theorem \ref{thm-inf} and Theorem \ref{thm-Kp}, we have that
    \[K^{H,p}_{\{G<a\}}(o)=\inf_{\xi\in\mathbf{S}_H}K_{\xi,\{G<a\},p}(o),\]
    and
    \[(-\infty,0]\ni a \to e^{(2n+pk)a}K_{\xi,\{G<a\},p}(o)\]
    is non-decreasing. It follows that
    \[e^{(2n+pk)a}K^{H,p}_{\{G<a\}}(o)=\inf_{\xi\in\mathbf{S}_H}e^{(2n+pk)a}K_{\xi,\{G<a\},p}(o)\]
    is also non-decreasing in $a\in (-\infty,0]$. Since
    \[e^{(2n+pk)a}K^{H,p}_{\{G<a\}}(o)=K^{H,p}_{\Omega_a}(o),\]
    we obtain that $K^{H,p}_{\Omega_a}(o)$ is non-decreasing in $a\in (-\infty,0]$.
\end{proof}

\begin{proof}[Proof of Corollary \ref{cor-BS.infinite.dim}]
    Combining Theorem \ref{thm-inf} and Theorem \ref{thm-Kp}, for $p\in [1,2]$, we have
    \[K^{H,p}_{\Omega}(o)=\inf_{\xi\in\mathbf{S}_H}K_{\xi,\Omega,p}(o)\ge K_{I_{\Omega}(o)}^{H,p}(o).\]

    If the space $A^p(I_{\Omega}(o))$ is infinitely dimensional, then there exists a sequence of homogeneous polynomials $\{H_j\}_{j\ge 1}$ of degree $k_j$, where $k_j$ increasingly tends to $+\infty$, such that $K^{H_j,p}_{I_{\Omega}(o)}>0$ for each $j$. Since $K^{H_j,p}_{\Omega}(o)\ge K_{I_{\Omega}(o)}^{H_j,p}(o)$, we have $K^{H_j,p}_{\Omega}(o)>0$ for each $j$, which implies $A^p(\Omega)$ is infinitely dimensional.
\end{proof}

\end{document}